\documentclass[11pt,a4paper,reqno]{amsart}
\usepackage{amsfonts,caption,amsmath,amssymb,enumerate,latexsym,xcolor,amsthm,multicol,cases,tikz,empheq,bm}
\usepackage{hyperref}
\usepackage[foot]{amsaddr}
\usetikzlibrary{calc}
\usetikzlibrary{decorations.pathreplacing}

\newcommand\AGL{\mathrm{AGL}}\newcommand\Aut{\mathrm{Aut}}\newcommand\Alt{\mathrm{Alt}}

\newcommand\FF{\mathbb{F}}\newcommand\Fix{\mathrm{Fix}}\newcommand\fpr{\mathrm{fpr}}

\newcommand\GL{\mathrm{GL}}

\newcommand\PSL{\mathrm{PSL}}

\newcommand\Sym{\mathrm{Sym}}

\newcommand\Z{\mathbb{Z}}

\newtheorem{theorem}{Theorem}[section]
\newtheorem{corollary}[theorem]{Corollary}
\newtheorem{lemma}[theorem]{Lemma}
\newtheorem{proposition}[theorem]{Proposition}
\theoremstyle{definition}
\newtheorem{definition}[theorem]{Definition}

\newtheorem{conjecture}[theorem]{Conjecture}

\newtheorem*{remark}{Remark}

\makeatletter
\renewcommand{\email}[2][]{%
  \ifx\emails\@empty\relax\else{\g@addto@macro\emails{,\space}}\fi%
  \@ifnotempty{#1}{\g@addto@macro\emails{\textrm{(#1)}\space}}%
  \g@addto@macro\emails{#2}%
}
\makeatother

\title{Determining the vertex stabilizers of $4$-valent half-arc-transitive graphs}%

\author{Binzhou Xia$^\dagger$}
\author{Zhishuo Zhang$^{\dagger,*}$}
\author{Sanming Zhou$^\dagger$}

\address{$^*$Corresponding author}
\address{$^\dagger$School of Mathematics and Statistics\\The University of Melbourne\\Parkville, VIC 3010\\Australia}
\email{binzhoux@unimelb.edu.au (Binzhou Xia), zhishuoz@student.unimelb.edu.au (Zhishuo Zhang), sanming@unimelb.edu.au (Sanming Zhou)}

\begin{document}

\begin{abstract}
We say that a group is a $4$-HAT-stabilizer if it is the vertex stabilizer of some connected $4$-valent half-arc-transitive graph. In 2001, Maru\v{s}i\v{c} and Nedela proved that every $4$-HAT-stabilizer must be a concentric group. However, over the past two decades, only a very small proportion of concentric groups have been shown to be $4$-HAT-stabilizers. This paper develops a theory that provides a general framework for determining whether a concentric group is a $4$-HAT-stabilizer. With this approach, we significantly extend the known list of $4$-HAT-stabilizers. As a corollary, we confirm that $\mathcal{H}_7\times C_2^{m-7}$ are $4$-HAT-stabilizers for $m\geq 7$, achieving the goal of a conjecture posed by Spiga and Xia.


\end{abstract}

\maketitle

\section{Introduction}\label{sec:introduction}

Throughout this paper, graphs are assumed to be finite, simple and undirected. We denote by $V(\Gamma)$, $E(\Gamma)$ or $A(\Gamma)$ the set of all vertices, edges, or arcs of a graph $\Gamma$, where \emph{arcs} are ordered pairs of adjacent vertices. We use $\Aut(\Gamma)$ to denote the full automorphism group of $\Gamma$. Let $G$ be a subgroup of $\mathrm{Aut}(\Gamma)$. We say that $G$ (or the graph $\Gamma$) is \emph{vertex-transitive}, \emph{edge-transitive} or \emph{arc-transitive} if 
$G$ (or $\Aut(\Gamma)$) acts transitively on $V(\Gamma)$, $E(\Gamma)$ or $A(\Gamma)$, respectively. We say that $G$ (or the graph $\Gamma$) is \emph{half-arc-transitive} if $G$ (or $\Aut(\Gamma)$) is vertex-transitive and  edge-transitive but not arc-transitive. For preliminaries on permutation group theory, particularly the types of primitive groups included in the O'Nan Scott Theorem, the reader is directed to~\cite{P1990}.

\subsection{Background}
In 1966, Tutte proved in his book \cite{T1966} that half-arc-transitive graphs must have even valency and noted that the existence of such graphs was, at the time, unknown. Clearly, there is no half-arc-transitive graph of valency $2$, so the smallest valency of half-arc-transitive groups is $4$. In 1970, Bouwer \cite{B1970} constructed a $2k$-valent half-arc-transitive graph for each $k\geq 2$, which completely answered Tutte's question. In 1998, Maru\v{s}i\v{c} stated in his survey paper \cite{M1998} that research on half-arc-transitive graphs gained new momentum in the 1990s, with ``thrilling" progress made in the study of $4$-valent half-arc-transitive graphs. As he anticipated, this field has been very active over the past two decades, especially on half-arc-transitive graphs of valency $4$.

In the study of vertex-transitive graphs, the vertex stabilizers are of great interest and crucial importance. In 2001, Maru\v{s}i\v{c} and Nedela \cite{MN2001} constrained the structure of vertex stabilizers of connected $4$-valent half-arc-transitive graphs to concentric groups, defined as follows.
\begin{definition}\cite[page~28]{MN2001}\label{def:concentric}
      A group $H=\langle a_1,\ldots,a_m\rangle$ is said to be \emph{concentric} if $|\langle a_i,\ldots,a_j\rangle|=2^{j-i+1}$ for all $i<j$ and there exists a group isomorphism $\varphi: \langle a_1,\ldots,a_{m-1}\rangle\rightarrow \langle a_2,\ldots,a_{m}\rangle$ such that $a_i^{\varphi}=a_{i+1}$ for each $i\in\{1,\ldots,m-1\}$.
\end{definition}
They showed that a group $H$ is a concentric group if and only if there exists a connected $4$-valent graph $\Gamma$ and a group $G\leq \Aut(\Gamma)$ such that $G$ is  half-arc-transitive and the vertex stabilizer of $G$ is isomorphic to $H$. In particular, the vertex stabilizers of $4$-valent half-arc-transitive graphs must be concentric groups. 
However, given that the half-arc-transitive action of $G$ does not imply $\Aut(\Gamma)$ is also half-arc-transitive, it is a challenging and long-standing problem to determine which concentric groups are indeed the vertex stabilizers of some connected $4$-valent half-arc-transitive graphs. 

For convenience, we call a group a \emph{$4$-HAT-stabilizer} if it is the vertex stabilizer of some connected $4$-valent half-arc-transitive graph. Despite considerable effort, only a small proportion of concentric groups have been identified as $4$-HAT-stabilizers. In fact, all the $4$-HAT-stabilizers known to date \cite{CM2003,CPS2015,M1998,M2005,S2016,SX2021,X2021} can be summarized as follows:
      \begin{equation}\label{eq:list}
            \begin{split}
            C^m_2 ~\text{for~all}~m\geqslant 1, &\quad D_8\times C_2^{m-3}~\text{for~all}~ m\geqslant 3, \\
            D^2_8\times C_2^{m-6} ~\text{for~all}~ m\geqslant 6, &\quad \mathcal{H}_7\times C_2^{m-7} ~\text{for}~ 7\leqslant m\leqslant 8,
            \end{split}
      \end{equation}
where $C_2^k$ is the direct product of $k$ copies of cyclic groups of order $2$, $D_8$ is the dihedral group of order $8$, and $\mathcal{H}_7$ is a group of order $2^7$ defined by
\begin{align}
    \mathcal{H}_7=\langle a_1,\dots,a_7 \mid ~&a_i^2=1 ~\text{for}~i\leqslant 7, (a_ia_j)^2=1~\text{for} ~|i-j|\leqslant 4,\notag\\
    &(a_1a_6)^2=a_3,~ (a_2a_7)^2=a_4, ~(a_1a_7)^2=a_5\rangle.\label{eq:H}
\end{align}
This list also implies the following theorem:
\begin{theorem}\cite[Theorem~1.4]{SX2021}\label{thm:2^8}
      Every concentric group of order at most $2^8$ is a $4$-HAT-stabilizer.
\end{theorem}

\subsection{The main result}
In this paper, we develop a general theory to determine which concentric groups are $4$-HAT-stabilizers (see Subsection~\ref{subsec:1.3}). By applying this theory, we significantly extend the above known list of $4$-HAT-stabilizers among concentric groups. To state the main result, we introduce an equivalent definition of concentric groups, with the equivalence established in~\cite[Theorem~5.5]{MN2001}\footnote{In \rm(R3) of ~\cite[Theorem~5.5]{MN2001}, the subscript $2d-2m+j-1$ should be $2d-2m+j-i$. This typo was first spotted and corrected by Poto\v{c}nik and Verret in~\cite[Page~499]{PV2010}.}.

\begin{definition}\label{def:rel}
    We call a group $H=\langle a_1,\ldots,a_m\rangle$  \emph{concentric} if there exist an integer $d$ with $2m/3\leq d\leq m$ and parameters 
    \begin{center}
          \begin{tabular}{llllll}
                $\varepsilon_{d,0}$& $\varepsilon_{d,1}$&$\cdots$& $\varepsilon_{d,3d-2m}$&&\\
                $\varepsilon_{d+1,0}$&$\varepsilon_{d+1,1}$&$\cdots$ &\quad $\cdots$& $\varepsilon_{d+1,3d-2m+1}$&\\
                \ \ $\vdots$&\ \ $\vdots$ & & &&$\ddots$\\
                $\varepsilon_{m-1,0}$&$\varepsilon_{m-1,1}$&$\cdots$& $\quad\cdots$& \ \ $\cdots$& $\varepsilon_{m-1,2d-m-1}$
          \end{tabular}
      \end{center}
    in $\FF_2$ such that the following conditions hold:
    \begin{itemize}
          \item $a_1^2=\dots=a_m^2=1$;
          \item $d$ is the largest integer such that $[a_i,a_j]=1$ whenever $|i-j|<d$;
          \item $[a_i,a_j]=a_{m-d+i}^{\varepsilon_{j-i,0}}\dots a_{d-m+j}^{\varepsilon_{j-i,2d-2m+j-i}}$ whenever $j-i\geq d$.
    \end{itemize}
\end{definition}
This definition clearly illustrates the structure of concentric groups through generators and relations. The parameter $d$ in the definition is referred to as the \emph{diameter} of $H$. It is evident that $H\cong C_2^m$ if $d=m$. Moreover, it is not hard to prove (see~\cite[Lemma~6.1]{MN2001}) that  $H\cong D_8\times C_2^{m-3}$ if $d=m-1$. Next we define tightly concentric groups. 

\begin{definition}\label{def:regular}
      Let $H=\langle a_1,\ldots,a_m\rangle$ be a concentric group with diameter $d$ and parameters $\varepsilon_{i,j}$ defined in Definition~\ref{def:rel}. We say that $H$ is a \emph{tightly concentric group} if $3d>2m$ and the following two conditions hold:
      \begin{align}
            &\varepsilon_{d,0}=1\ \text{ and }\ \varepsilon_{d,i}=0\ \text{ for each }\ i\in \{1,\ldots,3d-2m\};\label{enu:C1}\tag{C1}\\
            &\varepsilon_{m-i,2d-m-i}=0\ \text{ for each }\ i\in\{1,\ldots,m-d\}.\label{enu:C2}\tag{C2}
      \end{align}
\end{definition}

Arrange $\varepsilon_{i,j}$ as shown in Definition~\ref{def:rel}. Conditions~\eqref{enu:C1} and~\eqref{enu:C2} specify that the `top and right boundaries' consist entirely of zeros, except for $\varepsilon_{d,0}=1$. This indicates that the restriction for concentric groups to be tightly concentric is mild. The case $3d=2m$ must be excluded in Definition~\ref{def:regular} because condition~\eqref{enu:C1} would contradict condition~\eqref{enu:C2} in this situation.
With the tightly concentric group defined, we can now state our main theorem neatly.

\begin{theorem}\label{thm:regular}
      Every tightly concentric group is a $4$-HAT-stabilizer.
\end{theorem}

To illustrate the significance of this result in extending the list of known $4$-HAT-stabilizers among concentric groups, we present in the following Table~\ref{tab:1} the diameters $d$ and parameters $\varepsilon_{i,j}$ of groups in~\eqref{eq:list}.

\begin{table}[h!]
      \centering
      \setlength{\arrayrulewidth}{1pt}
      \renewcommand{\arraystretch}{1.1}
      \begin{tabular}{c|c|c} 
            \hline
            Group & $d$ & $\{(i,j)\mid \varepsilon_{i,j}=1\}$\\ [0.5ex] 
            \hline
            $C_2^m$ $(m\geq 1)$ &  $m$ & $\emptyset$\\ 
            $D_8\times C_2^{m-3}$ $(m\geq 3)$ & $m-1$ &$\{(d,0)\}$\\
            $D_8^2\times C_2^{m-6}$ $(m\geq 6)$ & $m-2$ & $\{(d,0)\}$\\
            $\mathcal{H}_7\times C_2^{m-7}$ $(m\in\{7,8\})$ &$m-2$ & $\{(d,0)$, $(d+1,2)\}$\\
            \hline
      \end{tabular}
      \caption{Diameter $d$ and parameters $\varepsilon_{i,j}$ of groups in~\eqref{eq:list}.}
      \label{tab:1}
\end{table}

Clearly, by applying Theorem~\ref{thm:regular}, we can extend the two examples $\mathcal{H}_7$ and $\mathcal{H}_7\times C_2$ in the last row of Table~\ref{tab:1} to the infinite family as follows, which accomplishes the goal of \cite[Conjecture~5.3]{SX2021}.
\begin{corollary}
      The group $\mathcal{H}_7\times C_2^{m-7}$ is a $4$-HAT-stabilizer for each $m\geq 7$.
\end{corollary}

As can be seen from Definition~\ref{def:rel}, the structure of the concentric group becomes more complex as $m-d$ increases. This may explain why all the known $4$-HAT-stabilizers in the literature have $m-d\leq 2$. Our main theorem is capable of overcoming the challenges associated with larger values of $m-d$, providing a powerful tool to identify complex $4$-HAT-stabilizers among concentric groups.

\subsection{The approach}\label{subsec:1.3}
Extending some earlier work \cite{CM2003,M2005,MN2001}, Spiga~\cite{S2016} formally introduced a strategy in 2016 for constructing connected $4$-valent half-arc-transitive graphs with a prescribed vertex stabilizer. This strategy successfully led to the identification of the $4$-HAT-stabilizers $C_2^m$~\cite[Theorem~1.1]{M2005}, $D_8^2$ and $\mathcal{H}_7$ \cite[Theorem~1.1]{S2016}, as well as $D_8^2\times C_2^{m-6}$ \cite[Theorem~1.2]{X2021}. Below, we outline the most effective portion of the strategy, starting from a concentric group $H$:
\begin{enumerate}[\textsf{Step} I.]
      \item \label{enu:stepI} Choose a permutation $x$ on $H$ and construct a $4$-valent graph $\Gamma$ such that $\langle R(H),x\rangle$ is half-arc-transitive with vertex stabilizer isomorphic to $H$, where $R(H)$ is the right regular permutation representation of $H$.
      \item \label{enu:stepII} Show that $\langle R(H),x\rangle=\Alt(H)$.
      \item \label{enu:stepIII} Show that $\langle R(H),x\rangle=\Aut(\Gamma)$. 
\end{enumerate}

Step~\ref{enu:stepI} can be accomplished either through coset graph construction (see, for example, \cite[Construction~3.3]{X2021}) or by taking the orbital graph of a transitive permutation group with a non-self-paired suborbit of length $2$ (see, for example, \cite[Theorem~1.1 and Section~2]{MN2001}). 

In~\cite{S2016}, Spiga completed Step~\ref{enu:stepII} for $D_8^2$ and $\mathcal{H}_7$ via direct verification using Magma~\cite{Magma} and suggested that it might also be proved using O'Nan-Scott Theorem. This was later confirmed by Xia~\cite{X2021} for the infinite family $D_8^2\times C_2^{m-6}$, but is still challenging for general concentric groups.

For Step~\ref{enu:stepIII}, the corresponding proofs for the groups considered in~\cite{S2016} and~\cite{X2021} involve very sophisticated analysis based on the O'Nan-Scott Theorem and the Classification of Finite Simple Groups, which are hard to generalize. In 2021, Spiga and Xia~\cite{SX2021} developed an alternative method to determine whether $H$ is a $4$-HAT-stabilizer, which eludes Step~\ref{enu:stepIII}. Precisely, they showed the following:
\begin{theorem}\cite[Theorem~1.1]{SX2021}\label{thm:2.2}
      Let $\Gamma$ be a connected $4$-valent graph, and let $T$ be a nonabelian simple half-arc-transitive subgroup of $\Aut(\Gamma)$. Then the vertex stabilizer of $T$ is a $4$-HAT-stabilizer. 
\end{theorem}

Once Steps~\ref{enu:stepI} and~\ref{enu:stepII} are completed, we obtain a connected $4$-valent graph $\Gamma$ such that  $\Alt(H)$ is half-arc-transitive on $\Gamma$ with vertex stabilizer isomorphic to $H$. Since $\Alt(H)$ is nonabelian simple when $|H|\geq 5$, we can immediately conclude by~Theorem~\ref{thm:2.2} that $H$ is a $4$-HAT-stabilizer. Following this idea, Spiga and Xia~\cite{SX2021} showed that $D_8\times C_2^{m-3}~(m\geq 3)$, $\mathcal{H}_7$ and $\mathcal{H}_7\times C_2$ are all $4$-HAT-stabilizers.

Now the remaining difficulty in determining whether a concentric group $H$ is a $4$-HAT-stabilizer lies in proving $\langle R(H),x\rangle=\Alt(H)$. In this paper, we introduce a theory that treats $x$ as an undetermined parameter (as suggested in \cite[Page 2143]{S2016}) and systematically explore the existence of $x$ such that $\langle R(H),x\rangle=\Alt(H)$. This theory establishes a general framework for determining whether a concentric group is a $4$-HAT-stabilizer. Using this approach, we successfully prove that every tightly concentric group is a $4$-HAT-stabilizer (Theorem~\ref{thm:regular}). Furthermore, this theory offers a promising foundation for proving that every concentric group is a $4$-HAT-stabilizer. Thereby, we put forward the following conjecture, which seems to be widely believed but has not been explicitly stated in the literature.

\begin{conjecture}\label{conj:1.8}
      Every concentric group is a $4$-HAT-stabilizer, and therefore, a group is a $4$-HAT-stabilizer if and only if it is concentric.
\end{conjecture}

For a concentric group $H=\langle a_1,\ldots,a_m\rangle$, let $\varphi$ be the group isomorphism defined in Definition~\ref{def:concentric}. According to Theorem~\ref{thm:2^8}, we may assume that $m\geq 9$. In the remainder of this section, we outline our general framework for proving that $H$ is a $4$-HAT-stabilizer, which consists of some established foundations and four guiding steps. 

For each $\tau \in a_1\langle a_2,\ldots,a_m\rangle$, define $x_{\tau}\in \Sym(H)$ by letting
\begin{equation}\label{eq:phi}
      h^{x_\tau}=h^{\varphi}\, \text{ and } \,(a_mh)^{x_\tau}=\tau h^{\varphi}\  \text{ for } h\in \langle a_1,\ldots,a_{m-1}\rangle.
\end{equation}
By Lemma~\ref{lm:condition}, we deduce from~\cite[Theorem~1.1]{MN2001} that for each $\tau \in a_1\langle a_2,\ldots,a_m\rangle$, there exists a $4$-valent graph such that $\langle R(H),x_\tau\rangle$ is half-arc-transitive with vertex stabilizer isomorphic to $H$. Then by Theorem~\ref{thm:2.2}, to show that $H$ is a $4$-HAT-stabilizer, we only need to prove that there exists $\tau\in a_1\langle a_2,\ldots,a_m\rangle$ such that $\langle R(H),x_\tau\rangle=\Alt(H)$. Since for each $\tau\in a_1\langle a_2,\ldots,a_m\rangle$, the group $\langle R(H),x_\tau\rangle$ is contained in $\Alt(H)$ (see Lemma~\ref{lm:SubAlt}) and is neither a primitive almost simple group nor a primitive group of diagonal type (see Lemma~\ref{lm:AS and Dia}), we proceed to complete the proof in the following four steps: 

\begin{enumerate}[\textsf{Step} 1.]
      \item \label{enu:step1} Find a subset $I_1\subseteq a_1\langle a_2,\ldots,a_m\rangle$ as large as possible such that $\langle R(H),x_\tau\rangle$ is a primitive permutation group on $H$ for each $\tau \in I_1$.
      \item \label{enu:step2} Find a subset $I_2\subseteq a_1\langle a_2,\ldots,a_m\rangle$ as large as possible such that no affine subgroup of $\Sym(H)$ contains $\langle R(H),x_\tau\rangle$ for any $\tau \in I_2$.
      \item \label{enu:step3} Find a subset $I_3\subseteq a_1\langle a_2,\ldots,a_m\rangle$ as large as possible such that $\langle R(H),x_\tau\rangle$ is not a subgroup of a wreath product in $\Sym(H)$ for any $\tau \in I_3$.
      \item \label{enu:step4} Verify that $I_1\cap I_2\cap I_3\neq \emptyset$ and conclude that there exists $\tau\in a_1\langle a_2,\ldots,a_m\rangle$ such that $\langle R(H),x_\tau\rangle=\Alt(H)$ by applying  all the results above and O'Nan-Scott Theorem.
\end{enumerate}


After the preparation in Sections~\ref{sec:preliminary} and~\ref{sec:calculation}, we demonstrate the implementation of Steps~\ref{enu:step1}--\ref{enu:step4} in Sections~\ref{sec:primitive}--\ref{sec:proof}, which finally leads to a proof of Theorem~\ref{thm:regular}. For concentric groups that are not tightly concentric, it is worth remarking that our general framework still applies; however, the implementation of Steps~\ref{enu:step1}--\ref{enu:step4} may be more intricate. Therefore, there is significant room for further research on this topic.

\section{Preliminaries}\label{sec:preliminary}
In this section, we set up the notation needed in this paper, present key properties of concentric groups, prove some foundational results of our theory and introduce coordinate forms of concentric groups.

\subsection{Notation}\label{subsec:2.1}
For a group $G$, let $G'$ denote the commutator subgroup of $G$, and  $Z(G)$ the center of $G$. Denote by $\Sym(\Omega)$ the symmetric group on the set $\Omega$. For $\pi\in\Sym(\Omega)$, define $\mathrm{Supp}(\pi)$ as the subset of $\Omega$ consisting of all the points that are not fixed by $\pi$. Denote by $\GL_n(q)$ the general linear group of degree $n$ over $\FF_q$ and let $\AGL_n(q)$ be the affine group on $\FF_q^n$. Use $\FF_2[w_1,\ldots,w_k]$ to represent the algebra over $\FF_2$ generated by the variables $w_1,\ldots, w_k$.

We fix the following notation throughout the paper. Let $H=\langle a_1,\ldots,a_m\rangle$ be a concentric group, let $d$ denote the diameter of $H$, and let 
\[
      d'=m-d.
\]
As mentioned after Definition~\ref{def:rel}, $H$ is isomorphic to $C_2^m$ or $D_8\times C_2^{m-3}$ if $d'=0$ or $1$, respectively. Such concentric groups $H$ are already known to be  $4$-HAT-stabilizers in~\cite{M1998} and~\cite{SX2021}. Thus we assume that $d'\geq 2$. Moreover, in view of Theorem~\ref{thm:2^8} (\cite[Theorem~1.4]{SX2021}), we assume that $m\geq 9$. Set
\begin{equation}\label{eq:H_{i,j}}
      H_{i,j}=\langle a_i,\ldots,a_j\rangle \  \text{ for }\ 1\leq i\leq j\leq m.
\end{equation}
We use $\varepsilon_{i,j}$ to represent the parameters of $H$, as given in Definition~\ref{def:rel}. Define
\[
      \varphi:H_{1,m-1}\to H_{2,m}
\]
as the group isomorphism mapping $a_i$ to $a_{i+1}$ for each $i\in \{1,\ldots,m-1\}$, and extend $\varphi$ to a permutation $\phi$ on $H$ by letting
\begin{equation}\label{eq:perm phi}
      h^\phi=h^\varphi\, \text{ and } \, (ha_m)^\phi=a_1h^\varphi\  \text{ for } h\in H_{1,m-1}.
\end{equation}

As in~\eqref{eq:phi}, for each $\tau\in a_1H_{2,m}$, we define $x_{\tau}$ to be the permutation on $H$ given by 
\[
      h^{x_\tau}=h^{\varphi}\, \text{ and } \,(a_mh)^{x_\tau}=\tau h^{\varphi}\  \text{ for } h\in H_{1,m-1}.
\]
Also, define a permutation $y_\tau$ on $H$ as follows:
\begin{equation}\label{eq:y}
      y_\tau=\big(x_\tau R(\tau)x_\tau^{-1} R(a_m)\big)^2.
\end{equation}
To simplify calculations, we assume $\tau\in a_1H_{2,d+1}$ for the rest of the paper unless specified otherwise (in fact, we only extend the range of $\tau$ to $a_1H_{2,m}$ in Subsection~\ref{subsec:2.3}).

\subsection{Concentric groups}
The following lemma gathers key properties of concentric groups, which we will apply repeatedly. For brevity, we will reference the lemma only when a reminder may benefit the reader.

\begin{lemma}\label{lm:concentric}
      The following statements hold:
      \begin{enumerate}[\rm(a)]
      \item \label{enu:0concentric} $|H_{i,j}|=2^{j-i-1}$ for $1\leq i\leq j\leq m$, and in particular, $|H|=2^m$;
      \item \label{enu:1concentric} $H'\leq Z(H)=H_{d'+1,d}\cong C_2^{d-d'}$, and $d-d'\geq d'$;
      \item \label{enu:2concentric} $[a_i,a_j]=[a_j,a_i]$ for each $i,j\in \{1,\ldots,m\}$;
      \item \label{enu:3concentric} for each $h_1,h_2,h_3\in H$, we have
      \[
            [h_1h_2,h_3]=[h_1,h_3][h_2,h_3]\ \text{ and }\ [h_1,h_2h_3]=[h_1,h_2][h_1,h_3];
      \]
      \item \label{enu:4concentric} every element of $H$ can be written as a word $a_1^{\alpha_1}\dots a_m^{\alpha_m}$
      for a unique $(\alpha_1,\ldots,\alpha_{m})\in \mathbb{F}_2^m$;
      \item \label{enu:5concentric} $\{h^2\mid h\in H\}\subseteq H'$;
      \item \label{enu:6concentric} the exponent of $H$ is $2$ or $4$. 
      \end{enumerate}
\end{lemma}

\begin{proof}
      Statement~\eqref{enu:0concentric} follows from Definition~\ref{def:concentric}, and statements~\eqref{enu:1concentric}--\eqref{enu:4concentric} can be easily obtained from Definition~\ref{def:rel}. Now we prove~\eqref{enu:5concentric}, which together with~\eqref{enu:1concentric} will imply~\eqref{enu:6concentric}. Take an arbitrary $h\in H$. By statement~\eqref{enu:4concentric}, we can write $h=a_1^{\alpha_1}\dots a_m^{\alpha_m}$ for some $(\alpha_1,\ldots,\alpha_{m})\in \mathbb{F}_2^m$. Since $H'\leq Z(H)$ and $a_i^2=1$ for each $i$, we obtain
      \[
            h^2=a_1^{\alpha_1}\dots a_m^{\alpha_m}a_1^{\alpha_1}\dots a_m^{\alpha_m}=a_1^{2\alpha_1}\dots a_m^{2\alpha_m}\prod_{i=1}^{m}\prod_{j=i+1}^m [a_i^{\alpha_i},a_j^{\alpha_j}]=\prod_{i=1}^{m}\prod_{j=i+1}^m [a_i^{\alpha_i},a_j^{\alpha_j}]\in H'.
      \]
      This gives $\{h^2\mid h\in H\}\subseteq H'$, completing the proof.
\end{proof}
Figure~\ref{fig:concentric} illustrates that the center $Z(H)$ of a concentric group $H$ is generated by $a_i$'s in the `middle section'. Since $d\geq 2m/3$, we have $|Z(H)|=2^{d-d'}=2^{2d-m}\geq 2^{m/3}=|H|^{1/3}$, indicating that concentric groups possess a large center.

\vspace{5pt}
\begin{figure}[h!]
      \centering
      \begin{tikzpicture}
            \filldraw 
            (0,0) circle (2pt) node [above] {$a_1$} -- 
            (3,0) circle (2pt) node [above] {$a_{d'}$} -- 
            (4,0) circle (2pt) node [below] {$a_{d'+1}$} -- 
            (9,0) circle (2pt) node [below] {$a_d$} -- 
            (10,0) circle (2pt) node [above] {$a_{d+1}$} --
            (13,0)circle (2pt) node [above] {$a_m$};
            \draw [decorate,decoration={brace,amplitude=7pt,mirror,raise=3.5ex}]
            (4,0) -- (9,0) node[midway,yshift=-3em]{$Z(H)$};
      \end{tikzpicture}
      \caption{Illustration of concentric groups}
      \label{fig:concentric}
\end{figure}
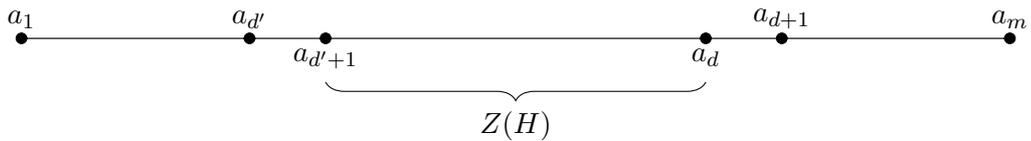

To help the reader follow the calculations in Section~\ref{sec:calculation}, we restate the three conditions in Definition~\ref{def:rel} alongside Figure~\ref{fig:concentric}, using the notation defined in Subsection~\ref{subsec:2.1}.
\begin{align}
      &a_1^2=\dots=a_m^2=1; \tag{R1}\label{enu:R1}\\
      &[a_i,a_j]=1\ \text{ whenever }\ |i-j|<d; \tag{R2}\label{enu:R2}\\
      &[a_i,a_j]=a_{d'+i}^{\varepsilon_{j-i,0}}\dots a_{j-d'}^{\varepsilon_{j-i,j-i-2d'}}\ \text{ whenever }\ i\in\{1,\ldots,d'\}\ \text{ and }\ j\in \{d+i,\ldots,m\}. \tag{R3}\label{enu:R3}
\end{align}

\subsection{Foundation of the theory}\label{subsec:2.3}
We now present the three lemmas referenced in Subsection~\ref{subsec:1.3}. 

\begin{lemma}\label{lm:condition}
      For each $\tau \in a_1H_{2,m}$ and $i\in \{1,\ldots,m-1\}$, we have $x_\tau^{-1}R(a_i)x_\tau=R(a_{i+1})$.
\end{lemma}

\begin{proof}
      Take an arbitrary $h\in H$. Then $h=\tau^{\alpha_1}h_1$ for some $\alpha_1\in \FF_2$ and $h_1\in H_{2,m}$, and thus,
      \[
            h^{x_\tau^{-1}R(a_i)x_\tau}
            =(\tau^{\alpha_1}h_1)^{x_\tau^{-1}R(a_i)x_\tau}
            =(a_m^{\alpha_1}h_1^{\varphi^{-1}}a_i)^{x_\tau}
            =\tau^{\alpha_1}(h_1^{\varphi^{-1}}a_i)^\varphi
            =\tau^{\alpha_1}h_1a_{i+1}
            =h^{R(a_{i+1})}.
      \]
      Hence $x_\tau^{-1}R(a_i)x_\tau=R(a_{i+1})$, as required.
\end{proof}

\begin{lemma}\label{lm:SubAlt}
      For each $\tau \in a_1H_{2,m}$, the group $\langle R(H),x_\tau\rangle$ is a subgroup of $\Alt(H)$.
\end{lemma}

\begin{proof}
      We first show that $R(H)\leq \Alt(H)$. For each $h\in H$, let $R(h)$ be the permutation on $H$ defined by $b^{R(h)}=bh$ for all $b\in H$. It follows that $R(h)$ is a product of $|H|/|h|$ disjoint cycles of length $|h|$, where $|h|$ is the order of $h$. Since $H$ is not a cyclic group and $|H|=2^m$, we obtain that $|H|/|h|$ is even, and thus, $R(H)\leq\Alt(H)$.

      Next we prove that $x_\tau$ is an even permutation on $H$. Let $L$ be the left multiplication permutation representation of $H$, and let $\zeta$ be the map on $H$ defined by 
      \[
            \zeta|_{H_{2,m}}=\mathrm{id}\quad \text{and}\quad \zeta|_{a_1H_{2,m}}=L(a_1\tau^{-1})|_{a_1H_{2,m}}.
      \]
      Clearly, $\zeta$ is a permutation on $H$ as $a_1\tau^{-1}\in H_{2,m}$. Similarly to the argument in the preceding paragraph, we conclude that $L(a_1\tau^{-1})|_{a_1H_{2,m}}$ is an even permutation on $a_1H_{2,m}$. Therefore $\zeta$ is an even permutation on $H$. Let $\xi$ be a permutation on $H$ defined by 
      \[
            h^{\xi}=h,\;\; (a_mh)^\xi=a_1h,\;\; (a_1h)^\xi=a_mh\;\; \text{and}\;\; (a_1a_mh)^\xi=a_ma_1h\ \text{ for }h\in H_{2,m-1}.
      \]
      Then $\xi|_{a_1H_{2,m-1}\cup a_mH_{2,m-1}}$ is a product of $|H_{2,m-1}|=2^{m-2}$ disjoint transpositions, and
      \[
           \xi|_{a_1a_mH_{2,m-1}}=L((a_ma_1)^2)|_{a_1a_mH_{2,m-1}} 
      \]
      is an even permutation on $a_1a_mH_{2,m-1}$. Hence $\xi$ is an even permutation on $H$. To complete the proof, it suffices to show that $w:=\xi x_\tau\zeta$ is an even permutation on $H$. In fact, a straightforward calculation gives that 
      \[
            h^w=h^\varphi,\;\; (a_mh)^w=a_2h^\varphi,\;\; (a_1h)^w=a_1h^\varphi\;\; \text{and}\;\; (a_1a_mh)^w=a_1a_2h^\varphi\ \text{ for }h\in H_{2,m-1},
      \]
      from which we observe that $w|_{H_{2,m}}$ and $w|_{a_1H_{2,m}}$ are permutations on $H_{2,m}$ and $a_1H_{2,m}$, respectively, with the same cycle type. Therefore, $w$ is an even permutation, as desired.
\end{proof}

\begin{lemma}\label{lm:AS and Dia}
      Suppose that $m\geq 9$. Then the subgroup $\langle R(H),x_\tau\rangle$ of $\Sym(H)$ is neither a primitive almost simple group nor a primitive group of diagonal type for any $\tau \in a_1H_{2,m}$.
\end{lemma}

\begin{proof}
      Write $G=\langle R(H),x_\tau\rangle$. First suppose that $G$ is a primitive almost simple group. Since $R(H)$ is a regular subgroup of $G$ and $|R(H)|$ is a power of $2$, we derive from~\cite[Theorem~1.1]{LPS2010} that $H$ is metacyclic (in this case, the socle of $G$ is $\PSL_2(q)$ for some Merssenne prime $q$). This in conjunction with Lemma~\ref{lm:concentric}\eqref{enu:6concentric} implies $|H|\leq 4^2$, contradicting that $|H|=2^m$ with $m\geq 9$. 

      Next suppose that $G$ is a primitive group of diagonal type. Then the degree of $G$ is $|T|^n$ for some nonabelian simple group $T$ and positive integer $n$. However, the degree of $G$ is $|H|=2^m$, a contradiction.
\end{proof}

\subsection{Coordinate forms}
For convenience in calculations, we now introduce the coordinate form of elements of $H$. Let $\FF_2^m$ be an $m$-dimensional vector space over $\FF_2$ with standard basis $\bm{e}_1,\ldots,\bm{e}_m$. By Lemma~\ref{lm:concentric}\eqref{enu:4concentric}, there is a bijection $H\to \FF_2^m$ mapping $a_1^{\alpha_1}\dots a_{m}^{\alpha_{m}}$ to $(\alpha_1,\ldots, \alpha_{m})$. We refer to $(\alpha_1,\ldots, \alpha_{m})\in \FF_2^m$ as the \emph{coordinate form} of $a_1^{\alpha_1}\dots a_{m}^{\alpha_{m}}$. Then $\bm{e}_i$ is the coordinate form of $a_i$ for $i\in \{1,\ldots,m\}$. Throughout this paper, we use $\bm{\tau}=(\tau_1,\ldots,\tau_m)$ to represent the coordinate form of the undetermined $\tau\in a_1H_{2,d+1}$. Then we have
\[
      \tau_1=1\ \text{ and }\ \tau_{d+2}=\dots=\tau_m=0,
\]
while $\tau_2,\ldots,\tau_{d+1}$ are undetermined parameters in $\FF_2$.

Let $\bm{\alpha}$ and $\bm{\beta}$ be the coordinate forms of $a_1^{\alpha_1}\dots a_m^{\alpha_m}$ and $a_1^{\beta_1}\dots a_m^{\beta_m}$, respectively. Denote by $\bm{\alpha}\bm{\beta}$ the coordinate form of $(a_1^{\alpha_1}\dots a_m^{\alpha_m})(a_1^{\beta_1}\dots a_m^{\beta_m})$, and $[\bm{\alpha},\bm{\beta}]$ the coordinate form of $[a_1^{\alpha_1}\dots a_m^{\alpha_m},a_1^{\beta_1}\dots a_m^{\beta_m}]$. For $g\in \Sym(H)$, we define $\bm{\alpha}^g$ as the coordinate form of $(a_1^{\alpha_1}\dots a_m^{\alpha_m})^g$. Then, recalling the definition of $\phi$ in~\eqref{eq:perm phi}, we have
\[
      \bm{\alpha}^\phi=(\alpha_m,\alpha_1,\ldots,\alpha_{m-1}).
\]
For the reader's convenience, we summarize the correspondence between elements of $H$ and their coordinate forms in Table~\ref{tab:correspondence}. 
\begin{table}[h!]
\centering
\setlength{\arrayrulewidth}{1pt}
\renewcommand{\arraystretch}{1.1}
\begin{tabular}{c|c} 
      \hline
      Elements of $H$ & Coordinate form in $\FF_2^m$ \\ [0.5ex] 
      \hline
      $a_1^{\alpha_1}\dots a_{m}^{\alpha_{m}}$ &  $\bm{\alpha}$\\ 
      $a_i$&$\bm{e}_i$\\
      $\tau$ & $\bm{\tau}=(1,\tau_2,\ldots,\tau_{d+1},0,\ldots,0)$\\
      $(a_1^{\alpha_1}\dots a_m^{\alpha_m})(a_1^{\beta_1}\dots a_m^{\beta_m})$  & $\bm{\alpha}\bm{\beta}$\\
      $[a_1^{\alpha_1}\dots a_m^{\alpha_m},a_1^{\beta_1}\dots {a_m^{\beta_m}}]$ & $[\bm{\alpha},\bm{\beta}]$\\
      $(a_1^{\alpha_1}\dots a_m^{\alpha_m})^g$ &$\bm{\alpha}^g$\\
      $(a_1^{\alpha_1}\dots a_m^{\alpha_m})^\phi$& $\bm{\alpha}^\phi=(\alpha_m,\alpha_1,\ldots,\alpha_{m-1})$\\
      \hline
\end{tabular}
\caption{Coordinate form correspondence}
\label{tab:correspondence}
\end{table}

For each $\ell\in\{d'+1,\ldots,d\}$, $\bm{\alpha}=(\alpha_1,\ldots,\alpha_m)\in \FF_2^m$ and $\bm{\beta}=(\beta_1,\ldots,\beta_m)\in \FF_2^m$, define
\begin{equation}\label{eq:convolution}
      \lambda_\ell(\bm{\alpha},\bm{\beta})=\sum_{i=1}^{\min\{d',\ell-d'\}}\sum_{j=\max\{d+i,\ell+d'\}}^m\beta_i\alpha_j\varepsilon_{j-i,\ell-d'-i}.
\end{equation}
In particular, we will frequently use the following special case:
\begin{equation}\label{eq:e_m,alpha}
      \lambda_\ell(\bm{e}_m,\bm{\beta})=\sum_{i=1}^{\min\{d',\ell-d'\}}\beta_i\varepsilon_{m-i,\ell-d'-i}. 
\end{equation}

The following lemma lists three straightforward observations regarding the notation $\lambda_\ell(\bm{\alpha},\bm{\beta})$.
\begin{lemma}\label{lm:properties}
      Let $\bm{\alpha}=(\alpha_1,\ldots,\alpha_m)$, $\bm{\beta}=(\beta_1,\ldots,\beta_m)$ and $\bm{\gamma}=(\gamma_1,\ldots,\gamma_m)$ be elements of $\FF_2^m$. Then the following hold:
      \begin{enumerate}[\rm(a)]
            \item \label{enu:bilinear} (Bilinearity) for each $r,s\in \FF_2$, we have 
            \begin{align*}
                  \lambda_\ell(r\bm{\alpha}+s\bm{\gamma},\bm{\beta})&=r\lambda_\ell(\bm{\alpha},\bm{\beta})+s\lambda_\ell(\bm{\gamma},\bm{\beta}),\\
                  \lambda_\ell(\bm{\alpha},r\bm{\beta}+s\bm{\gamma})&=r\lambda_\ell(\bm{\alpha},\bm{\beta})+s\lambda_\ell(\bm{\alpha},\bm{\gamma});
            \end{align*}
            \item \label{enu:algebra} $\lambda_\ell(\bm{\alpha},\bm{\beta})$ is contained in the algebra $\FF_2[\alpha_{\max\{d+i,\ell+d'\}},\ldots,\alpha_m,\beta_1,\ldots,\beta_{\min\{d',\ell-d'\}}]$;
            \item \label{enu:domain} if $\alpha_i=\beta_i$ for each $i\in \{1,\ldots,\min\{d',\ell-d'\}\}$, then $\lambda_\ell(\bm{e}_m,\bm{\alpha})=\lambda_\ell(\bm{e}_m,\bm{\beta})$.
      \end{enumerate}
\end{lemma}

\section{Formulas and technical lemmas}\label{sec:calculation}

\subsection{Formulas}

Based on Lemma~\ref{lm:concentric}, the following lemma follows immediately from conditions~\eqref{enu:R2} and~\eqref{enu:R3}.

\begin{lemma}\label{lm:e_i,e_j}
      For $1\leq i\leq j\leq m$, we have
      \[
            [\bm{e}_i,\bm{e}_j]=[\bm{e}_j,\bm{e}_i]=
            \begin{cases}
                  \sum\limits_{\ell=d'+i}^{j-d'}\varepsilon_{j-i,\ell-d'-i}\bm{e}_\ell& \ \textup{ if }\ 1\leq i\leq d'\ \textup{ and }\ d+1\leq j \leq m\\
                  \bm{0}& \ \textup{ otherwise.}
            \end{cases} 
      \]
\end{lemma}

The next two lemmas establish formulas for $\bm{\alpha}\bm{\beta}$ and $[\bm{\alpha},\bm{\beta}]$ involving $\lambda_\ell(\bm{\alpha},\bm{\beta})$.

\begin{lemma}\label{lm:multiplication}
      For $\bm{\alpha},\bm{\beta}\in \FF_2^m$, we have
      \[
            \bm{\alpha}\bm{\beta}=\bm{\alpha}+\bm{\beta}+\sum_{\ell=d'+1}^d\lambda_\ell(\bm{\alpha},\bm{\beta})\bm{e}_\ell.
      \]
\end{lemma}

\begin{proof}
      Let $\bm{\alpha}=(\alpha_1,\ldots,\alpha_m)$ and $\bm{\beta}=(\beta_1,\ldots,\beta_m)$. It is a straightforward calculation that 
      \begin{align*}
            (a_1^{\alpha_1}\dots a_m^{\alpha_m})(a_1^{\beta_1}\dots a_m^{\beta_m})
            &=a_1^{\alpha_1+\beta_1}\dots a_m^{\alpha_m+\beta_m}\prod_{i=1}^{d'}\prod_{j=d+i}^m[a_i^{\beta_i},a_j^{\alpha_j}]\\
            &=a_1^{\alpha_1+\beta_1}\dots a_m^{\alpha_m+\beta_m}\prod_{i=1}^{d'}\prod_{j=d+i}^m[a_i,a_j]^{\beta_i\alpha_j}.
      \end{align*}
      This together with Lemma~\ref{lm:e_i,e_j} shows that the coordinate form of $(a_1^{\alpha_1}\dots a_m^{\alpha_m})(a_1^{\beta_1}\dots a_m^{\beta_m})$ is 
      \begin{align*}
            \bm{\alpha}\bm{\beta}
            &=\bm{\alpha}+\bm{\beta}+\sum_{i=1}^{d'}\sum_{j=d+i}^m\beta_j\alpha_i[\bm{e}_i,\bm{e}_j]\\
            &=\bm{\alpha}+\bm{\beta}+\sum_{i=1}^{d'}\sum_{j=d+i}^m\sum_{\ell=d'+i}^{j-d'}\beta_j\alpha_i\varepsilon_{j-i,\ell-d'-i}\bm{e}_\ell\\
            &=\bm{\alpha}+\bm{\beta}+\sum_{\ell=d'+1}^d\sum_{i=1}^{\min\{d',\ell-d'\}}\sum_{j=\max\{d+i,\ell+d'\}}^m\beta_i\alpha_j\varepsilon_{j-i,\ell-d'-i}\bm{e}_\ell\\
            &=\bm{\alpha}+\bm{\beta}+\sum_{\ell=d'+1}^d\lambda_\ell(\bm{\alpha},\bm{\beta})\bm{e}_\ell.\qedhere
      \end{align*}
\end{proof}

\begin{lemma}\label{lm:commutator}
      For $\bm{\alpha},\bm{\beta}\in \FF_2^m$, we have
      \[
            [\bm{\alpha},\bm{\beta}]=[\bm{\beta},\bm{\alpha}]=\sum_{\ell=d'+1}^d(\lambda_\ell(\bm{\alpha},\bm{\beta})+\lambda_\ell(\bm{\beta},\bm{\alpha}))\bm{e}_\ell.
      \]
      In particular, the following holds for $\bm{\alpha}\in \FF_2^m$:
      \[
            [\bm{\alpha},\bm{e}_m]=[\bm{e}_m,\bm{\alpha}]=\sum_{\ell=d'+1}^d\lambda_\ell(\bm{e}_m,\bm{\alpha})\bm{e}_\ell.
      \]
\end{lemma}

\begin{proof}
      Let $\bm{\alpha}=(\alpha_1,\ldots,\alpha_m)$ and $\bm{\beta}=(\beta_1,\ldots,\beta_m)$. Then
      \begin{align*}
            [a_1^{\alpha_1}\dots a_m^{\alpha_m},a_1^{\beta_1}\dots a_m^{\beta_m}]
            =\prod_{i=1}^m\prod_{j=1}^m[a_i^{\alpha_i},a_j^{\alpha_j}]
            &=\Big(\prod_{i=1}^{d'}\prod_{j=d+i}^m[a_i,a_j]^{\alpha_i\beta_j}\Big)\Big(\prod_{j=1}^{d'}\prod_{i=d+j}^m[a_i,a_j]^{\alpha_i\beta_j}\Big)\\
            &=\Big(\prod_{i=1}^{d'}\prod_{j=d+i}^m[a_i,a_j]^{\alpha_i\beta_j}\Big)\Big(\prod_{i=1}^{d'}\prod_{j=d+i}^m[a_j,a_i]^{\alpha_j\beta_i}\Big).
      \end{align*}
      Now the first conclusion follows from a similar calculation to the proof of Lemma~\ref{lm:multiplication}. To see the second conclusion, notice from Lemma~\ref{lm:properties}\eqref{enu:algebra} that $\lambda_\ell(\bm{\alpha},\bm{e}_m)=0$ for each $\ell\in\{d'+1,\ldots,d\}$.
\end{proof}

The rest of this subsection is devoted to the calculation of $\bm{\alpha}^{x_\tau}$, $\bm{\alpha}^{x_\tau^{-1}}$, $\bm{\alpha}^{y_\tau}$, $\bm{\alpha}^{y_\tau^{x_\tau^t}}$ for all $t\in \{-d'+1,\ldots,d-d'+1\}$.

\begin{lemma}\label{lm:x coordinate form}
      Let $\bm{\alpha}=(\alpha_1,\ldots,\alpha_m)\in\FF_2^m$. Then 
      \[
            \bm{\alpha}^{x_\tau}=\bm{\alpha}^\phi+\alpha_m\Big(\bm{e}_1+\bm{\tau}+\sum_{\ell=d'+2}^{d+1}\lambda_{\ell-1}(\bm{e}_m,\bm{\alpha})\bm{e}_{\ell}\Big).
      \]
\end{lemma}

\begin{proof}
      Let $h=a_1^{\alpha_1}\dots a_{m-1}^{\alpha_{m-1}}$ and $\bm{h}=(\alpha_1,\ldots,\alpha_{m-1},0)$. If $\alpha_m=0$, then 
      \[
            (a_1^{\alpha_1}\dots a_{m-1}^{\alpha_{m-1}}a_m^{\alpha_m})^{x_\tau}=h^{x_\tau}=h^\varphi=h^\phi,
      \]
      and thus $\bm{\alpha}^{x_\tau}=\bm{h}^\phi=\bm{\alpha}^\phi$, as desired. Now assume $\alpha_m=1$. Then $\bm{\alpha}^{x_\tau}$ is the coordinate form of
      \begin{equation}\label{eq:eqx}
            (a_1^{\alpha_1}\dots a_{m-1}^{\alpha_{m-1}}a_m)^{x_\tau}=(ha_m)^{x_\tau}=(a_mh[h,a_m])^{x_\tau}=\tau h^\varphi[h,a_m]^\varphi.
      \end{equation}
      In view of $\bm{\tau}=(1,\tau_2,\ldots,\tau_{d+1},0,\ldots,0)$ and $\bm{h}^\phi=(0,\alpha_1,\ldots,\alpha_{m-1})$, we calculate by~\eqref{eq:convolution} that $\lambda_\ell(\bm{\tau},\bm{h}^\phi)=0$ for each $\ell\in\{d'+1,\ldots,d\}$.
      Hence, by Lemma~\ref{lm:multiplication}, the coordinate form of $\tau h^\varphi$ is 
      \[
            \bm{\tau}\bm{h}^\phi=\bm{\tau}+\bm{h}^\phi+\sum_{\ell=d'+1}^d\lambda_\ell(\bm{\tau},\bm{h}^\phi)\bm{e}_\ell=\bm{\tau}+\bm{h}^\phi=\bm{\tau}+\bm{\alpha}^\phi+\bm{e}_1.
      \]
      Moreover, it follows from Lemmas~\ref{lm:commutator} and~\ref{lm:properties}\eqref{enu:domain} that the coordinate form of $[h,a_m]^\varphi$ is 
      \[
            [\bm{h},\bm{e}_m]^\phi
            =\sum_{\ell=d'+1}^d\lambda_\ell(\bm{e}_m,\bm{h})\bm{e}_{\ell+1}
            =\sum_{\ell=d'+1}^d\lambda_\ell(\bm{e}_m,\bm{\alpha})\bm{e}_{\ell+1}
            =\sum_{\ell=d'+2}^{d+1}\lambda_{\ell-1}(\bm{e}_m,\bm{\alpha})\bm{e}_{\ell}.
      \]
      Combining these formulas of $\bm{\tau}\bm{h}^\phi$ and $[\bm{h},\bm{e}_m]^\phi$ with~\eqref{eq:eqx} and noting that the first $d'$ coordinates of $[\bm{h},\bm{e}_m]^\phi$ are all zero, we obtain by Lemmas~\ref{lm:multiplication} and~\ref{lm:properties}\eqref{enu:algebra} that
      \[
            \bm{\alpha}^{x_\tau}
            =(\bm{\tau}\bm{h}^\phi)[\bm{h},\bm{e}_m]^\phi
            =\bm{\tau}\bm{h}^\phi+[\bm{h},\bm{e}_m]^\phi
            =\bm{\tau}+\bm{\alpha}^\phi+\bm{e}_1+\sum_{\ell=d'+2}^{d+1}\lambda_{\ell-1}(\bm{e}_m,\bm{\alpha})\bm{e}_{\ell}.
      \]
      This completes the proof.
\end{proof}

\begin{lemma}\label{lm:x^-1 coordinate form}
      Let $\bm{\alpha}=(\alpha_1,\ldots,\alpha_m)\in \FF_2^m$. Then 
      \[
            \bm{\alpha}^{x_\tau^{-1}}=\bm{\alpha}^{\phi^{-1}}+\alpha_1\Big(\bm{\tau}^{\phi^{-1}}+\bm{e}_m+\sum_{\ell=d'+1}^{d}\lambda_\ell(\bm{e}_m,\bm{\alpha}^{\phi^{-1}}+\bm{\tau}^{\phi^{-1}})\bm{e}_\ell\Big).
      \]
\end{lemma}

\begin{proof}
      Write $a_1^{\alpha_1}\dots a_{m}^{\alpha_{m}}=\tau^{\alpha_1} h$ for some $h\in H_{2,m}$, and let $\bm{h}$ be the coordinate form of $h$. If $\alpha_1=0$, then $(a_1^{\alpha_1}\dots a_{m}^{\alpha_{m}})^{x_\tau^{-1}}=h^{x_\tau^{-1}}=h^{\varphi^{-1}}=h^\phi$, and so $\bm{\alpha}^{x_\tau^{-1}}=\bm{h}^{x_\tau^{-1}}=\bm{h}^{\phi^{-1}}=\bm{\alpha}^{\phi^{-1}}$, as desired. Now assume $\alpha_1=1$. Then $\bm{\alpha}^{x_\tau^{-1}}$ is the coordinate form of
      \[
            (a_1^{\alpha_1}\dots a_{m}^{\alpha_{m}})^{x_\tau^{-1}}=(\tau h)^{x_\tau^{-1}}=a_m h^{\varphi^{-1}}=h^{\varphi^{-1}}a_m[a_m,h^{\varphi^{-1}}].
      \]
      This combined with Lemmas~\ref{lm:multiplication},~\ref{lm:properties}\eqref{enu:algebra}, and~\ref{lm:commutator} gives that
      \begin{align*}
            \bm{\alpha}^{x_\tau^{-1}}
            &=(\bm{h}^{\phi^{-1}}\bm{e}_m)[\bm{e}_m,\bm{h}^{\phi^{-1}}]\\
            &=(\bm{h}^{\phi^{-1}}+\bm{e}_m)\Big(\sum_{\ell=d'+1}^d \lambda_\ell(\bm{e}_m,\bm{h}^{\phi^{-1}})\bm{e}_\ell\Big)
            =\bm{h}^{\phi^{-1}}+\bm{e}_m+\sum_{\ell=d'+1}^d \lambda_\ell(\bm{e}_m,\bm{h}^{\phi^{-1}})\bm{e}_\ell.
      \end{align*}
      Noting that $a_1^{\alpha_1}\dots a_{m}^{\alpha_{m}}=\tau h$ and $h\in H_{2,m}$, we derive from Lemma~\ref{lm:multiplication} and~\eqref{eq:convolution} that $\bm{\alpha}=\bm{\tau}\bm{h}=\bm{\tau}+\bm{h}$, and so $\bm{h}^{\phi^{-1}}=\bm{\alpha}^{\phi^{-1}}+\bm{\tau}^{\phi^{-1}}$. Therefore,
      \[
            \bm{\alpha}^{x_\tau^{-1}}=\bm{\alpha}^{\phi^{-1}}+\bm{\tau}^{\phi^{-1}}+\bm{e}_m+\sum_{\ell=d'+1}^d \lambda_\ell(\bm{e}_m,\bm{\alpha}^{\phi^{-1}}+\bm{\tau}^{\phi^{-1}})\bm{e}_\ell.
      \]
      This completes the proof.
\end{proof}

\begin{lemma}\label{lm:C2}
      Suppose that $\eqref{enu:C2}$ holds. Then $\lambda_d(\bm{\alpha},\bm{\beta})=0$ for each $\bm{\alpha},\bm{\beta}\in\FF_2^m$, and 
      \[
            [a_m,h]\in H_{d'+1,d-1}\ \text{ and }\ [a_m,h]^\varphi\in H_{d'+2,d}\leq Z(H)
      \]
      for each $h\in H$.
\end{lemma}

\begin{proof}
      Since $d-d'\geq d'$, we derive from~\eqref{eq:convolution} and~\eqref{enu:C2} that, for each 
      $\bm{\alpha}=(\alpha_1,\ldots,\alpha_m)$ and $\bm{\beta}=(\beta_1,\ldots,\beta_m)$ in $\FF_2^m$, we have
      \[
            \lambda_d(\bm{\alpha},\bm{\beta})=\sum_{i=1}^{d'}\beta_i\alpha_m\varepsilon_{m-i,d-d'-i}=0.
      \]
      For each $h\in H$ with coordinate form $\bm{h}$, the above result together with Lemma~\ref{lm:commutator} implies that 
      \[
            [\bm{e_m},\bm{h}]=\sum_{\ell=d'+1}^d\lambda_\ell(\bm{e}_m,\bm{h})\bm{e}_\ell=\sum_{\ell=d'+1}^{d-1}\lambda_\ell(\bm{e}_m,\bm{h})\bm{e}_\ell,
      \]
      which shows $[a_m,h]\in H_{d'+1,d-1}$. As a consequence, $[a_m,h]^\varphi\in  H_{d'+1,d-1}^\varphi=H_{d'+2,d}\leq Z(H)$.
\end{proof}

\begin{lemma}\label{lm:y}
      Suppose that $\eqref{enu:C1}$ and $\eqref{enu:C2}$ hold, and let $\bm{\alpha}=(\alpha_1,\ldots,\alpha_m)\in \FF_2^m$. Then
      \[
            \bm{\alpha}^{y_\tau}=\bm{\alpha}+
            \big(\sum_{i=d}^{m-1}\alpha_i\varepsilon_{i,0}+\alpha_m\tau_{d+1}\big)\bm{e}_{2d'}.
      \]
\end{lemma}

\begin{proof}
      Let $h=a_1^{\alpha_1}\dots a_{m-1}^{\alpha_{m-1}}$, $b=[a_m,h]$ and $c=[\tau,h^\varphi]$. We first prove that 
      \begin{equation}\label{eq:preY}
            (a_1^{\alpha_1}\dots a_m^{\alpha_{m}})^{y_\tau}= h [a_m,c^{\varphi^{-1}}] [a_m,(\tau^2)^{\varphi^{-1}}]^{\alpha_m} a_m^{\alpha_m}.
      \end{equation}

      Assume $\alpha_m=0$. It is straightforward to deduce from the definition of $x_\tau$ that
      \[
            h^{x_\tau R(\tau)x_\tau^{-1} R(a_m)}
            =(h^\varphi\tau)^{x_\tau^{-1}R(a_m)}
            =(\tau h^\varphi c)^{x_\tau^{-1}R(a_m)}
            =a_m h c^{\varphi^{-1}}a_m
            =hc^{\varphi^{-1}}[a_m,c^{\varphi^{-1}}]b.
      \]
      Recall from~\eqref{eq:y} that $y_\tau=\big(x_\tau R(\tau)x_\tau^{-1} R(a_m)\big)^2$. Thus $(a_1^{\alpha_1}\dots a_m^{\alpha_{m}})^{y_\tau}$ is equal to 
      \[
            h^{y_\tau}=(hc^{\varphi^{-1}}[a_m,c^{\varphi^{-1}}]b)^{x_\tau R(\tau)x_\tau^{-1} R(a_m)}
            =(h^\varphi c [a_m,c^{\varphi^{-1}}]^\varphi b^\varphi \tau)^{x_\tau^{-1} R(a_m)}.
      \]
      Since $b,c\in Z(H)$ and Lemma~\ref{lm:C2} implies that $[a_m,c^{\varphi^{-1}}]^\varphi, b^\varphi\in Z(H)$, we obtain that 
      \begin{align*}
            h^{y_\tau}
            &=(h^\varphi c [a_m,c^{\varphi^{-1}}]^\varphi b^\varphi \tau)^{x_\tau^{-1} R(a_m)}\\
            &=(\tau h^\varphi [a_m,c^{\varphi^{-1}}]^\varphi b^\varphi)^{x_\tau^{-1} R(a_m)}=a_mh[a_m,c^{\varphi^{-1}}] ba_m=h[a_m,c^{\varphi^{-1}}].
      \end{align*}
      This proves~\eqref{eq:preY} for the case $\alpha_m=0$. 

      Now assume $\alpha_m=1$. Since $b^\varphi\in Z(H)$, we obtain
      \begin{align*}
            (ha_m)^{x_\tau R(\tau)x_\tau^{-1} R(a_m)}
            &=(a_mhb)^{x_\tau R(\tau)x_\tau^{-1} R(a_m)}\\
            &=(\tau h^\varphi
            b^\varphi\tau)^{x_\tau^{-1} R(a_m)}=\big(\tau^2h^\varphi c
            b^\varphi\big)^{x_\tau^{-1} R(a_m)}
            =(\tau^2)^{\varphi^{-1}}h c^{\varphi^{-1}}ba_m,
      \end{align*}
      and thus $(a_1^{\alpha_1}\dots a_m^{\alpha_{m}})^{y_\tau}$ is equal to 
      \begin{align*}
            (ha_m)^{y_\tau}
            &=\big((\tau^2)^{\varphi^{-1}}h c^{\varphi^{-1}}ba_m\big)^{x_\tau R(\tau)x_\tau^{-1} R(a_m)}\\
            &=\big(a_m (\tau^2)^{\varphi^{-1}}[a_m,(\tau^2)^{\varphi^{-1}}] h c^{\varphi^{-1}} [a_m,c^{\varphi^{-1}}]\big)^{x_\tau R(\tau)x_\tau^{-1} R(a_m)}\\
            &=\big(\tau^3[a_m,(\tau^2)^{\varphi^{-1}}]^\varphi h^\varphi c [a_m,c^{\varphi^{-1}}]^\varphi\tau\big)^{x_\tau^{-1} R(a_m)}.
      \end{align*}
      Note by Lemma~\ref{lm:C2} that $[a_m,(\tau^2)^{\varphi^{-1}}]^\varphi$, $[a_m,c^{\varphi^{-1}}]^\varphi\in Z(H)$ and by Lemma~\ref{lm:concentric}(e) that $\tau^4=1$. These combined with $c\in Z(H)$ give that
      \begin{align*}
            (ha_m)^{y_\tau}
            &=\big(\tau^4[a_m,(\tau^2)^{\varphi^{-1}}]^\varphi h^\varphi [a_m,c^{\varphi^{-1}}]^\varphi\big)^{x_\tau^{-1} R(a_m)}\\
            &=\big([a_m,(\tau^2)^{\varphi^{-1}}]^\varphi h^\varphi [a_m,c^{\varphi^{-1}}]^\varphi\big)^{x_\tau^{-1} R(a_m)}=[a_m,(\tau^2)^{\varphi^{-1}}] h [a_m,c^{\varphi^{-1}}] a_m,
      \end{align*}
      completing the proof of~\eqref{eq:preY}.

      Next we calculate $[a_m,c^{\varphi^{-1}}]$ and $[a_m,(\tau^2)^{\varphi^{-1}}]^{\alpha_m}$. It follows from Lemma~\ref{lm:e_i,e_j} and~\eqref{enu:C1} that 
      \[
            [\bm{e}_{d'},\bm{e}_m]=\sum_{\ell=2d'}^d\varepsilon_{d,\ell-2d'}\bm{e}_\ell=\bm{e}_{2d'},
      \]
      which means that $[a_{d'},a_m]=a_{2d'}$. Let $\omega$ be the $(d'+1)$-th coordinate of $c=[\tau,h^\varphi]$. Since $c\in Z(H)=H_{d'+1,d}$, we have $c^{\varphi^{-1}}\in H_{d',d-1}$, and thus 
      \[
            [a_m,c^{\varphi^{-1}}]=[a_m,a_{d'}]^\omega=[a_{d'},a_m]^\omega=a_{2d'}^\omega.
      \]
      Let $\bm{h}=(\alpha_1,\ldots,\alpha_{m-1},0)$. It follows from Lemma~\ref{lm:commutator},~\eqref{eq:convolution} and $d\geq 2d'$ that
      \[
            \omega=\lambda_{d'+1}(\bm{\tau},\bm{h}^\phi)+\lambda_{d'+1}(\bm{h}^\phi,\bm{\tau})
            =\sum_{j=d+1}^m \tau_1 \alpha_{j-1} \varepsilon_{j-1,0}
            =\sum_{j=d}^{m-1} \alpha_j \varepsilon_{j,0}.
      \]
      Recalling from Lemma~\ref{lm:concentric}\eqref{enu:5concentric} that $\tau^2\in H'\leq H_{d'+1,d}$, we deduce from Lemma~\ref{lm:multiplication} that the coordinate form of $(\tau^2)^{\varphi^{-1}}=(\tau^2)^{\phi^{-1}}$ is 
      \[
            (\bm{\tau}\bm{\tau})^{\phi^{-1}}=\sum\limits_{\ell=d'+1}^d \lambda_\ell(\bm{\tau},\bm{\tau})\bm{e}_{\ell-1}.
      \]
      Hence we derive from $\tau^2\in H_{d'+1,d}$ and~\eqref{eq:convolution} that
      \[
            [a_m,(\tau^2)^{\varphi^{-1}}]^{\alpha_m}=[a_m,a_{d'}]^{\lambda_{d'+1}(\bm{\tau},\bm{\tau})\alpha_m}=a_{2d'}^{\alpha_m\lambda_{d'+1}(\bm{\tau},\bm{\tau})}=a_{2d'}^{\alpha_m\tau_{d+1}}.
      \]
      
      In summary, viewing $[a_m,c^{\varphi^{-1}}]\in Z(H)$ and $[a_m,(\tau^2)^{\varphi^{-1}}]^{\alpha_m}\in Z(H)$, we conclude that the coordinate form of $(a_1^{\alpha_1}\dots a_m^{\alpha_{m}})^{y_\tau}= h [a_m,c^{\varphi^{-1}}] [a_m,(\tau^2)^{\varphi^{-1}}]^{\alpha_m} a_m^{\alpha_m}$ is 
      \[
            \bm{\alpha}+\big(\omega+\alpha_m\tau_{d+1}\big)\bm{e}_{2d'}
            =\bm{\alpha}+\big(\sum_{i=d}^{m-1}\alpha_i\varepsilon_{i,0}+\alpha_m\tau_{d+1}\big)\bm{e}_{2d'}.
      \]
      This completes the proof.
\end{proof}

To study the coordinate form of $y_\tau^{x_\tau^t}$, we define a sequence of functions $f_t: \FF_2^m\times \FF_2^m\to \FF_2$ $(t\in \Z)$ by letting
\begin{equation}\label{eq:regular_f_i}
      f_0(\bm{\alpha},\bm{\tau})=\sum_{i=d}^{m-1}\alpha_i\varepsilon_{i,0}+\alpha_m\tau_{d+1}\ \text{ and }\  f_t(\bm{\alpha},\bm{\tau})=f_0(\bm{\alpha}^{x_\tau^{-t}},\bm{\tau}) \ \text{ for each}\ t\in\Z.
\end{equation}
We will write $f_t(\bm{\alpha},\bm{\tau})$ as $f_t(\bm{\alpha})$ for brevity when $\tau$ is fixed.

\begin{lemma}\label{lm:regular_y^{x^t}}
    Suppose that $\eqref{enu:C1}$ and $\eqref{enu:C2}$ hold, and let $f_t$ be as in~\eqref{eq:regular_f_i}. Then for each $\bm{\alpha}
    \in\FF_2^m$ and fixed $\tau\in a_1H_{2,d+1}$, the following hold:         
    \begin{align}
          \bm{\alpha}^{x_\tau^{-t}y_\tau x_\tau^t}&=
          \bm{\alpha}+f_t(\bm{\alpha}) \bm{e}_{2d'+t}\ \text{ for each }\ t\in\{-d'+1,\ldots,d-d'\},\label{eq:regular_y^{x^t}}\\
          \bm{\alpha}^{x_\tau^{-(d-d'+1)}y_\tau x_\tau^{d-d'+1}}&=
          \bm{\alpha}+f_{d-d'+1}(\bm{\alpha})\Big(\bm{\tau}+\sum_{\ell=d'+2}^d\lambda_{\ell-1}(\bm{e}_m,\bm{\alpha}^{x_\tau^{-1}})\bm{e}_{\ell}\Big).\label{eq:regular_y^{x^{2d-m+1}}}
    \end{align}
\end{lemma}

\begin{proof}
    Since $\tau$ is fixed, we abbreviate $x_\tau$ and $y_\tau$ to $x$ and $y$, respectively. We first use induction on $t$ to show that~\eqref{eq:regular_y^{x^t}} holds for $t\in \{0,\ldots,d-d'\}$. It follows from Lemma~\ref{lm:y} and~\eqref{eq:regular_f_i} that $\bm{\alpha}^{y}=
      \bm{\alpha}+f_0(\bm{\alpha}) \bm{e}_{2d'}$,
    which proves~\eqref{eq:regular_y^{x^t}} for $t=0$. Assume inductively that~\eqref{eq:regular_y^{x^t}} holds for some $t\in \{0,\ldots, d-d'-1\}$, which means 
    \[
      \bm{\alpha}^{x^{-t}yx^t}=\bm{\alpha}+f_t(\bm{\alpha})\bm{e}_{2d'+t}.
    \]
    This together with~\eqref{eq:regular_f_i} shows that
    \[
      \bm{\alpha}^{x^{-(t+1)}y x^{t+1}}
      =(\bm{\alpha}^{x^{-1}})^{(x^{-t}yx^t)x}=\big(\bm{\alpha}^{x^{-1}}+f_t(\bm{\alpha}^{x^{-1}})\bm{e}_{2d'+t}\big)^{x}=\big(\bm{\alpha}^{x^{-1}}+f_{t+1}(\bm{\alpha})\bm{e}_{2d'+t}\big)^{x}.
    \]
    Since $2d'+t\notin\{1,\ldots,d'\}\cup\{m\}$, we then deduce from Lemmas~\ref{lm:x coordinate form} and~\ref{lm:properties}\eqref{enu:algebra} that
    \[
      \bm{\alpha}^{x^{-(t+1)}yx^{t+1}}
      =(\bm{\alpha}^{x^{-1}})^x+f_{t+1}(\bm{\alpha})\bm{e}_{2d'+t+1}
      =\bm{\alpha}+f_{t+1}(\bm{\alpha})\bm{e}_{2d'+t+1},
    \]
    which completes the induction.

    Next, we apply another induction on $t$ to prove that~\eqref{eq:regular_y^{x^t}} holds for $t\in \{-d'+1,\ldots,0\}$. The base case $t=0$ is proved in the preceding paragraph. Assume inductively that~\eqref{eq:regular_y^{x^t}} holds for some $t\in \{-d'+2,\ldots, 0\}$, which means
    \[
      \bm{\alpha}^{x^{-t}yx^t}=\bm{\alpha}+f_t(\bm{\alpha})\bm{e}_{2d'+t}.
    \]
    This combined with~\eqref{eq:regular_f_i} gives that 
    \begin{align*}
      \bm{\alpha}^{x^{-(t-1)}yx^{t-1}}
      =(\bm{\alpha}^x)^{(x^{-t}yx^t)x^{-1}}=\big(\bm{\alpha}^x+f_t(\bm{\alpha}^x)\bm{e}_{2d'+t}\big)^{x^{-1}}=\big(\bm{\alpha}^x+f_{t-1}(\bm{\alpha})\bm{e}_{2d'+t}\big)^{x^{-1}}.
    \end{align*}
    Since $2d'+t\notin\{1,\ldots,d'+1\}$, we derive from Lemmas~\ref{lm:x^-1 coordinate form} and~\ref{lm:properties}\eqref{enu:algebra} that
    \[
      \bm{\alpha}^{x^{-(t-1)}yx^{t-1}}
      =(\bm{\alpha}^x)^{x^{-1}}+f_{t-1}(\bm{\alpha})\bm{e}_{2d'+t-1}
      =\bm{\alpha}+f_{t-1}(\bm{\alpha})\bm{e}_{2d'+t-1},
    \]
    This completes the induction and then the proof of~\eqref{eq:regular_y^{x^t}}.

    Finally, we verify~\eqref{eq:regular_y^{x^{2d-m+1}}}. It follows from~\eqref{eq:regular_y^{x^t}} that 
    \[
          \bm{\alpha}^{x^{-(d-d')}yx^{d-d'}}=
          \bm{\alpha}+f_{d-d'}(\bm{\alpha}) \bm{e}_m.
    \]
    This along with~\eqref{eq:regular_f_i} shows that 
    \begin{align*}
      \bm{\alpha}^{x^{-(d-d'+1)}yx^{d-d'+1}}
      &=(\bm{\alpha}^{x^{-1}})^{(x^{-(d-d')}yx^{d-d'})x}\\
      &=\big(\bm{\alpha}^{x^{-1}}+f_{d-d'}(\bm{\alpha}^{x^{-1}})\bm{e}_m\big)^{x}
      =\big(\bm{\alpha}^{x^{-1}}+f_{d-d'+1}(\bm{\alpha})\bm{e}_m\big)^{x}.
    \end{align*}
    Then we deduce from Lemmas~\ref{lm:x coordinate form} and~\ref{lm:properties}\eqref{enu:domain} that $\bm{\alpha}^{x^{-(d-d'+1)}yx^{d-d'+1}}$ is equal to
    \begin{align*}
      &\quad\,\,(\bm{\alpha}^{x^{-1}})^x+f_{d-d'+1}(\bm{\alpha})\bm{e}_1+f_{d-d'+1}(\bm{\alpha})\Big(\bm{e}_1+\bm{\tau}+\sum_{\ell=d'+2}^{d+1} \lambda_{\ell-1}\big(\bm{e}_m,\bm{\alpha}^{x^{-1}}+f_{d-d'+1}(\bm{\alpha})\bm{e}_m\big)\bm{e}_{\ell} \Big)\\
      &=\bm{\alpha}+f_{d-d'+1}(\bm{\alpha})\Big(\bm{\tau}+\sum_{\ell=d'+2}^{d+1} \lambda_{\ell-1}(\bm{e}_m,\bm{\alpha}^{x^{-1}})\bm{e}_{\ell} \Big).
    \end{align*}
    This combined with Lemma~\ref{lm:C2} completes the proof.
\end{proof}

\begin{corollary}\label{cor:fpiff}
      Suppose that $\eqref{enu:C1}$ and $\eqref{enu:C2}$ hold, let $t\in\{-d'+1,\ldots,d-d'+1\}$, and let $f_t$ be as in~\eqref{eq:regular_f_i}. Then $\bm{\alpha}\in \FF_2^m$ is a fixed point of $x_\tau^{-t}y_\tau x_\tau^t$ if and only if $f_t(\bm{\alpha},\bm{\tau})=0$.
\end{corollary}

\begin{proof}
    By~\eqref{eq:regular_y^{x^t}}, the conclusion holds for $t\in\{-d'+1,\ldots,d-d'\}$. Since $\tau_1=1$, it follows from~\eqref{eq:regular_y^{x^{2d-m+1}}} that the statement also holds for $t=d-d'+1$.
\end{proof}

\subsection{Technical lemmas}
In this subsection, we prove four technical lemmas to be used in Section~\ref{sec:PA}.

\begin{lemma}\label{lm:f_0,f_1}
      Suppose that $\eqref{enu:C1}$ and $\eqref{enu:C2}$ hold, and let $f_t$ be as in~\eqref{eq:regular_f_i}. Assume that $f_{d-d'+1}(\bm{\gamma},\bm{\tau})=1$ for some $\bm{\gamma}=(\gamma_1,\ldots,\gamma_m)$ and $\bm{\tau}=(1,\tau_2,\ldots,\tau_{d+1},0,\ldots,0)$ in $\FF_2^m$. Then the following hold:
      \begin{align*}
            f_0(\bm{\gamma}^{x_\tau^{-(d-d'+1)}y_\tau x_\tau^{d-d'+1}},\bm{\tau})&=f_0(\bm{\gamma},\bm{\tau})+\tau_d+\lambda_{d-1}(\bm{e}_m,\bm{\gamma}^{x_\tau^{-1}})+\tau_{d+1}\varepsilon_{d+1,0}+\gamma_m\tau_{d+1},\\
            f_{-1}(\bm{\gamma}^{x_\tau^{-(d-d'+1)}y_\tau x_\tau^{d-d'+1}},\bm{\tau})
            &=f_{-1}(\bm{\gamma},\bm{\tau})+\tau_{d-1}+\lambda_{d-2}(\bm{e}_m,\bm{\gamma}^{x_\tau^{-1}})+\big(\tau_d+\lambda_{d-1}(\bm{e}_m,\bm{\gamma}^{x_\tau^{-1}})\big)\varepsilon_{d+1,0}\\
            &+\tau_{d+1}\varepsilon_{d+2,0}+\gamma_m\lambda_{d-1}(\bm{e}_m,\bm{\tau}).
      \end{align*}
\end{lemma}

\begin{proof}
      For convenience, write $z_\tau=x_\tau^{-(d-d'+1)}y_\tau x_\tau^{d-d'+1}$.
      Since $f_{d-d'+1}(\bm{\gamma},\bm{\tau})=1$, it follows from~\eqref{eq:regular_y^{x^{2d-m+1}}} that
      \begin{equation}\label{eq:alpha^z}
          \bm{\gamma}^{z_\tau}
          =\bm{\gamma}+\bm{\tau}+\sum_{\ell=d'+2}^d\lambda_{\ell-1}(\bm{e}_m,\bm{\gamma}^{x_\tau^{-1}})\bm{e}_\ell.
      \end{equation}
      Recall from~\eqref{eq:regular_f_i} that 
      \[
            f_0(\bm{\gamma},\bm{\tau})=\sum_{i=d}^{m-1}\gamma_i\varepsilon_{i,0}+\gamma_m\tau_{d+1}.
      \]
      This combined with~\eqref{eq:alpha^z} gives
      \[
            f_0(\bm{\gamma}^{z_\tau},\bm{\tau})=\sum_{i=d}^{m-1}\gamma_i\varepsilon_{i,0}+\big(\tau_d+\lambda_{d-1}(\bm{e}_m,\bm{\gamma}^{x_\tau^{-1}})\big)\varepsilon_{d,0}+\tau_{d+1}\varepsilon_{d+1,0}+\gamma_m\tau_{d+1}.
      \]
      Since~\eqref{enu:C1} requires that $\varepsilon_{d,0}=1$, we obtain
      \[
            f_0(\bm{\gamma}^{z_\tau},\bm{\tau})=f_0(\bm{\gamma},\bm{\tau})+\tau_d+\lambda_{d-1}(\bm{e}_m,\bm{\gamma}^{x_\tau^{-1}})+\tau_{d+1}\varepsilon_{d+1,0}+\gamma_m\tau_{d+1},
      \]
      verifying the first conclusion.

      Since~\eqref{enu:C1} and~\eqref{enu:C2} hold, it follows from~\eqref{eq:regular_f_i}, Lemmas~\ref{lm:x coordinate form} and~\ref{lm:C2} that 
      \begin{align*}
            f_{-1}(\bm{\gamma},\bm{\tau})
            &=f_0(\bm{\gamma}^{x_\tau},\bm{\tau})\\
            &=f_0\Big(\bm{\gamma}^\phi+\gamma_m\Big(\bm{e}_1+\bm{\tau}+\sum_{\ell=d'+2}^{d+1}\lambda_{\ell-1}(\bm{e}_m,\bm{\gamma})\bm{e}_{\ell}\Big),\bm{\tau}\Big)\\
            &=f_0(\bm{\gamma}^\phi,\bm{\tau})+\gamma_m f_0(\bm{e}_1+\bm{\tau},\bm{\tau})+\gamma_mf_0\Big(\sum_{\ell=d'+2}^{d+1}\lambda_{\ell-1}(\bm{e}_m,\bm{\gamma})\bm{e}_{\ell},\bm{\tau}\Big)\\
            &=\sum_{i=d}^{m-1}\gamma_{i-1}\varepsilon_{i,0}+\gamma_{m-1}\tau_{d+1}+\gamma_m\big(\tau_d+\lambda_{d-1}(\bm{e}_m,\bm{\gamma})\big)\varepsilon_{d,0}+\gamma_m\big(\tau_{d+1}+\lambda_d(\bm{e}_m,\bm{\gamma})\big)\varepsilon_{d+1,0}\\
            &=\sum_{i=d-1}^{m-2}\gamma_i\varepsilon_{i+1,0}+\gamma_{m-1}\tau_{d+1}+\gamma_m\big(\tau_d+\lambda_{d-1}(\bm{e}_m,\bm{\gamma})\big)+\gamma_m\tau_{d+1}\varepsilon_{d+1,0}.
      \end{align*}
      This in conjunction with~\eqref{eq:alpha^z} and $\varepsilon_{d,0}=1$ yields that
      \begin{align*}
            f_{-1}(\bm{\gamma}^{z_\tau},\bm{\tau})
            &=f_{-1}\Big(\bm{\gamma}+\bm{\tau}+\sum_{\ell=d'+2}^d\lambda_{\ell-1}(\bm{e}_m,\bm{\gamma}^{x_\tau^{-1}})\bm{e}_\ell,\bm{\tau}\Big)\\
            &=\sum_{i=d-1}^{m-2}\gamma_i\varepsilon_{i+1,0}+\gamma_m\big(\tau_d+\lambda_{d-1}(\bm{e}_m,\bm{\gamma}^{z_\tau})\big)+\gamma_{m-1}\tau_{d+1}+\tau_{d-1}+\lambda_{d-2}(\bm{e}_m,\bm{\gamma}^{x_\tau^{-1}})\\
            &+\big(\gamma_m\tau_{d+1}+\tau_d+\lambda_{d-1}(\bm{e}_m,\bm{\gamma}^{x_\tau^{-1}})\big)\varepsilon_{d+1,0}+\tau_{d+1}\varepsilon_{d+2,0}.
      \end{align*}
      Note by parts~\eqref{enu:domain} and~\eqref{enu:bilinear} of Lemma~\ref{lm:properties} that 
      \[
            \lambda_{d-1}(\bm{e}_m,\bm{\gamma}^{z_\tau})=\lambda_{d-1}(\bm{e}_m,\bm{\gamma}+\bm{\tau})=\lambda_{d-1}(\bm{e}_m,\bm{\gamma})+\lambda_{d-1}(\bm{e}_m,\bm{\tau}).
      \]
      Therefore,
      \begin{align*}
            f_{-1}(\bm{\gamma}^{z_\tau},\bm{\tau})
            &=f_{-1}(\bm{\gamma},\bm{\tau})+\tau_{d-1}+\lambda_{d-2}(\bm{e}_m,\bm{\gamma}^{x^{-1}})+\big(\tau_d+\lambda_{d-1}(\bm{e}_m,\bm{\gamma}^{x_\tau^{-1}})\big)\varepsilon_{d+1,0}\\
            &+\tau_{d+1}\varepsilon_{d+2,0}+\gamma_m\lambda_{d-1}(\bm{e}_m,\bm{\tau}).\qedhere
      \end{align*}
\end{proof}

If~\eqref{enu:C1} holds, then define $u$ to be the largest in $\{d,\ldots,m-1\}$ such that $\varepsilon_{u,0}=1$, that is,
\begin{equation}\label{eq:u}
      u=\max\big\{i \mid i\in \{d,\ldots,m-1\},\ \varepsilon_{i,0}=1\big\}.
\end{equation}

\begin{lemma}\label{lm:expressions of f_t}
      Suppose that $\eqref{enu:C1}$ and $\eqref{enu:C2}$ hold, let $f_t$ be as in ~\eqref{eq:regular_f_i}, let $u$ be as in~\eqref{eq:u}, and let $\bm{\alpha}=(\alpha_1,\ldots,\alpha_m)$. Regarding $f_t(\bm{\alpha},\bm{\tau})$ as an element of the algebra $\FF_2[\alpha_1,\ldots,\alpha_m,\tau_2,\ldots,\tau_{d+1}]$, we have the following:
      \begin{enumerate}[\rm(a)]
            \item \label{enu:f_t1} for each $t\in\{1,\ldots,m-u\}$, 
            \[
                  f_t(\bm{\alpha},\bm{\tau})\in \alpha_{u+t}+\FF_2[\alpha_{d+1},\ldots,\alpha_{u+t-1}];
            \]
            \item \label{enu:f_t2} for each $t\in\{m-u+1,\ldots,d-d'+1\}$,
            \[
                  f_t(\bm{\alpha},\bm{\tau})\in \alpha_{t-m+u}+ \FF_2[\alpha_1,\ldots,\alpha_{t-m+u-1},\alpha_{d+1},\ldots,\alpha_m,\tau_2,\ldots,\tau_{t-m+u}];
            \]
            \item \label{enu:f_t3} for each $t\in\{-d'+1,\ldots,0\}$, 
            \[
                  f_t(\bm{\alpha},\bm{\tau})\in \alpha_{d+t}+\FF_2[\alpha_1,\ldots,\alpha_{d'},\alpha_{d+1+t},\ldots,\alpha_m,\tau_2,\ldots,\tau_{d+1}].
            \]
      \end{enumerate}
\end{lemma}

\begin{proof}
      Let $\bm{\beta}=(\beta_1,\ldots,\beta_m)=\bm{\alpha}^x$ and $\bm{\gamma}=(\gamma_1,\ldots,\gamma_m)=\bm{\alpha}^{x^{-1}}$. In the following, we take subscripts modulo $m$ for convenience. It follows from Lemmas~\ref{lm:x coordinate form} and~\ref{lm:x^-1 coordinate form} that 
      \begin{align}
            \beta_i
            &=\begin{cases}
                  \alpha_{i-1}& \ \text{ if }\ i\in \{1\}\cup \{d+2,\ldots,m\}\\
                  \alpha_{i-1}+\alpha_m\tau_i& \ \text{ if }\ i\in \{2,\ldots,d'+1\}\\
                  \alpha_{i-1}+\alpha_m(\tau_i+\lambda_{i-1}(\bm{e}_m,\bm{\alpha}))& \ \text{ if }\ i\in \{d'+2,\ldots,d+1\},
            \end{cases}\label{eq:[i]x}\\
            \gamma_i
            &=\begin{cases}
                  \alpha_{i+1}& \ \text{ if }\ i\in \{d+1,\ldots,m\}\\
                  \alpha_{i+1}+\alpha_1\tau_{i+1}& \ \text{ if }\ i\in \{1,\ldots,d'\}\\
                  \alpha_{i+1}+\alpha_1(\tau_{i+1}+\lambda_i(\bm{e}_m,\bm{\alpha}^{\phi^{-1}}+\bm{\tau}^{\phi^{-1}}))& \ \text{ if }\ i\in \{d'+1,\ldots,d\}.
            \end{cases}\label{eq:[i]x-1}
      \end{align}
      Moreover, we see from Lemma~\ref{lm:properties}\eqref{enu:algebra} that 
      \begin{align}
            &\lambda_{i-1}(\bm{e}_m,\bm{\alpha})\in \FF_2[\alpha_1,\ldots,\alpha_{d'}]\  \text{ for all }\ i\in \{d'+2\ldots,d+1\},\label{eq:lambda_range2}\\
            &\lambda_i(\bm{e}_m,\bm{\alpha}^{\phi^{-1}}+\bm{\tau}^{\phi^{-1}})\in \FF_2[\alpha_2,\ldots,\alpha_{i+1-d'},\tau_2,\ldots,\tau_{i+1-d'}]\  \text{ for all }\ i\in \{d'+1\ldots,d\}.\label{eq:lambda_range1}
      \end{align}
      We use induction on $t$ to prove each part the lemma.

      \textsf{Part~\eqref{enu:f_t1}.}
      Recall from~\eqref{eq:regular_f_i} that
      \[
            f_1(\bm{\alpha},\bm{\tau})
            =f_0(\bm{\gamma},\bm{\tau})
            =\sum_{i=d}^{m-1}\gamma_i\varepsilon_{i,0}+\gamma_m\tau_{d+1}.
      \]
      Then it is a straightforward calculation from~\eqref{eq:[i]x-1} and Lemma~\ref{lm:C2} that 
      \begin{align}
            f_1(\bm{\alpha},\bm{\tau})
            &=\sum_{i=d}^u\alpha_{i+1}\varepsilon_{i,0}+\alpha_1\big(\tau_{d+1}+\lambda_d(\bm{e}_m,\bm{\alpha}^{\phi^{-1}}+\bm{\tau}^{\phi^{-1}})\big)+\alpha_1\tau_{d+1}\nonumber\\
            &=\sum_{i=d}^u\alpha_{i+1}\varepsilon_{i,0}
            =\alpha_{u+1}+\sum_{i=d}^{u-1}\alpha_{i+1}\varepsilon_{i,0}
            \in\alpha_{u+1}+\FF_2[\alpha_{d+1},\ldots,\alpha_u].\label{eq:f_1}
      \end{align}
      Hence, the conclusion for $t=1$ holds. Suppose inductively that for $t\in \{1,\ldots,m-u-1\}$, we have 
      $f_t(\bm{\alpha})\in \alpha_{u+t}+\FF_2[\alpha_{d+1},\ldots,\alpha_{u+t-1}]$. Then it follows from~\eqref{eq:[i]x-1} that 
      \begin{align*}
            f_{t+1}(\bm{\alpha},\bm{\tau})
            =f_t(\bm{\gamma},\bm{\tau})
            &\in \gamma_{u+t}+\FF_2[\gamma_{d+1},\ldots,\gamma_{u+t-1}]\\
            &=\alpha_{u+t+1}+\FF_2[\alpha_{d+2},\ldots,\alpha_{u+t}]
            \subseteq \alpha_{u+t+1}+\FF_2[\alpha_{d+1},\ldots,\alpha_{u+t}].
      \end{align*}
      This completes the induction.

      \textsf{Part~\eqref{enu:f_t2}.}
      We see from part~\eqref{enu:f_t1} that $f_{m-u}(\bm{\alpha},\bm{\tau})\in \alpha_m+\FF_2[\alpha_{d+1},\ldots,\alpha_{m-1}]$. Then it follows from~\eqref{eq:[i]x-1} that 
      \begin{align*}
            f_{m-u+1}(\bm{\alpha},\bm{\tau})
            =f_{m-u}(\bm{\gamma},\bm{\tau})
            &\in \gamma_m+\FF_2[\gamma_{d+1},\ldots,\gamma_{m-1}]\\
            &=\alpha_1+\FF_2[\alpha_{d+2},\ldots,\alpha_m]
            \subseteq \alpha_1+\FF_2[\alpha_{d+1},\ldots,\alpha_m].
      \end{align*}
      This shows that the conclusion holds for $t=m-u+1$. Suppose inductively that for some $t\in\{m-u+1,\ldots,d-d'+1\}$, we have 
      \[
            f_t(\bm{\alpha},\bm{\tau})\in \alpha_{t-m+u}+ \FF_2[\alpha_1,\ldots,\alpha_{t-m+u-1},\alpha_{d+1},\ldots,\alpha_m,\tau_2,\ldots,\tau_{t-m+u}].
      \]
      Then we deduce from~\eqref{eq:[i]x-1} and~\eqref{eq:lambda_range1} that
      \begin{align*}
            f_{t+1}(\bm{\alpha},\bm{\tau})
            =f_t(\bm{\gamma},\bm{\tau})
            &\in \gamma_{t-m+u}+\FF_2[\gamma_1,\ldots,\gamma_{t-m+u-1},\gamma_{d+1},\ldots,\gamma_m,\tau_2,\ldots,\tau_{t-m+u}]\\
            &\subseteq \alpha_{t-m+u+1}+\FF_2[\alpha_1,\ldots,\alpha_{t-m+u},\alpha_{d+1},\ldots,\alpha_m,\tau_2,\ldots,\tau_{t-m+u+1}],
      \end{align*}
      completing the induction.
      
      \textsf{Part~\eqref{enu:f_t3}.}
      We see from~\eqref{enu:C1} and~\eqref{eq:regular_f_i} that the conclusion holds for $t=1$ as 
      \[
            f_0(\bm{\alpha},\bm{\tau})=\alpha_d+\alpha_m\tau_{d+1}+\sum_{i=d+1}^{m-1}\alpha_i\varepsilon_{i,0}\in \alpha_d+\FF_2[\alpha_1,\ldots,\alpha_{d'},\alpha_{d+1},\ldots,\alpha_m,\tau_2,\ldots,\tau_{d+1}].
      \]
      Similarly, suppose inductively that $ f_t(\bm{\alpha},\bm{\tau})\in \alpha_{d+t}+\FF_2[\alpha_1,\ldots,\alpha_{d'},\alpha_{d+1+t},\ldots,\alpha_m,\tau_2,\ldots,\tau_{d+1}]$ for some $t\in\{-d'+2,\ldots,0\}$. Then we derive from~\eqref{eq:[i]x} and~\eqref{eq:lambda_range2} that
      \begin{align*}
            f_{t-1}(\bm{\alpha},\bm{\tau})
            =f_t(\bm{\beta},\bm{\tau})
            &\in \beta_{d+t}+\FF_2[\beta_1,\ldots,\beta_{d'},\beta_{d+1+t},\ldots,\beta_m,\tau_2,\ldots,\tau_{d+1}]\\
            &\subseteq \alpha_{d+t-1}+\FF_2[\alpha_1,\ldots,\alpha_{d'},\alpha_{d+t},\ldots,\alpha_m,\tau_2,\ldots,\tau_{d+1}],
      \end{align*}
      which completes the induction.
\end{proof}

\begin{lemma}\label{lm:findGamma}
      Suppose that $\eqref{enu:C1}$ and $\eqref{enu:C2}$ hold. Let $3d>2m$, and let $f_t$ be as in ~\eqref{eq:regular_f_i}. Then for each $(\tau_2,\ldots,\tau_{d-4})\in\FF_2^{d-5}$, there exists $\bm{\gamma}\in \FF_2^m$ with $\gamma_1=\gamma_m=0$ such that $f_{d-d'+1}(\bm{\gamma},\bm{\tau})=1$ for each $(\tau_{d-3},\ldots,\tau_{d+1})\in\FF_2^5$.
\end{lemma}

\begin{proof}
      Let $u=\max\big\{i \mid i\in\{d,\ldots,m-1\},\ \varepsilon_{i,0}=1\big\}$ as defined in~\eqref{eq:u}. To prove the lemma, fix an arbitrary $(\tau_2,\ldots,\tau_{d-4})\in\FF_2^{d-5}$.
      
      Assume first that $u-2d'+1\geq 3$. Let $\bm{\alpha}=(0,0,\alpha_3,\ldots,\alpha_m)\in\FF_2^m$. Then we derive from Lemma~\ref{lm:x^-1 coordinate form} that 
      \begin{equation}\label{eq:20}
            \bm{\alpha}^{x_\tau^{-2}}=(\alpha_3,\ldots,\alpha_m,0,0).
      \end{equation}
      Since $u-2d'+1\geq 3$ implies that $d-d'-1\geq m-u+1$, we deduce from Lemma~\ref{lm:expressions of f_t}\eqref{enu:f_t2} that 
      \[
            f_{d-d'-1}(\bm{\alpha},\bm{\tau})\in \alpha_{u-2d'-1}+\FF_2[\alpha_1,\ldots,\alpha_{u-2d'-2},\alpha_{d+1},\ldots,\alpha_m,\tau_2,\ldots,\tau_{u-2d'-1}].
      \]
      This together with~\eqref{eq:regular_f_i} and~\eqref{eq:20} implies that 
      \begin{align}
            f_{d-d'+1}(\bm{\alpha},\bm{\tau})
            &=f_{d-d'-1}(\bm{\alpha}^{x_\tau^{-2}},\bm{\tau})\nonumber\\
            &\in \alpha_{u-2d'+1}+\FF_2[\alpha_3,\ldots,\alpha_{u-2d'},\alpha_{d+3},\ldots,\alpha_m,\tau_2,\ldots,\tau_{u-2d'-1}].\label{eq:21}
      \end{align}
      Since $d'\geq 2$, we have
      \[
            u-2d'-1\leq m-1-2d'-1=d-d'-2\leq d-4.
      \]
      Therefore, noting that $u-2d'+1\geq 3$, we derive from~\eqref{eq:21} that for each 
      \[
            (\alpha_3,\ldots,\alpha_{u-2d'},\alpha_{u-2d'+2},\ldots,\alpha_m)\in\FF_2^{m-3},
      \]
      there exists $\alpha_{u-2d'+1}\in \FF_2$ such that $f_{d-d'+1}(\bm{\alpha},\bm{\tau})=1$ for each  $(\tau_{d-3},\ldots,\tau_{d+1})\in \FF_2^5$. As a consequence, there exists $\bm{\gamma}\in \FF_2^m$ with $\gamma_1=\gamma_m=0$ such that $f_{d-d'+1}(\bm{\gamma},\bm{\tau})=1$ for each $(\tau_{d-3},\ldots,\tau_{d+1})\in\FF_2^5$.
      
      Assume next that $u-2d'+1< 3$. Recall that $u\geq d$ and by $3d>2m$ that $d>2d'$. It follows that $u-2d'+1\geq d-2d'+1\geq 2$, which forces $d-2d'=1$ and $u=d$. In this case, letting $\bm{\alpha}=(\alpha_1,\ldots,\alpha_m)\in\FF_2^m$, we obtain from Lemma~\ref{lm:expressions of f_t}\eqref{enu:f_t2} that 
      \[
            f_{d-d'+1}(\bm{\alpha},\bm{\tau})\in\alpha_2+\FF_2[\alpha_1,\alpha_{d+1},\ldots,\alpha_m,\tau_2].
      \]
      Hence, for each $(\alpha_1,\alpha_3,\alpha_4\ldots,\alpha_m)\in\FF_2^{m-1}$, there exists $\alpha_2\in \FF_2$ such that $f_{d-d'+1}(\bm{\alpha},\bm{\tau})=1$ for each  $(\tau_{d-3},\ldots,\tau_{d+1})\in \FF_2^5$. Consequently, for each $(\alpha_1,\alpha_3,\alpha_4\ldots,\alpha_m)\in\FF_2^{m-1}$, there exists $\bm{\gamma}\in \FF_2^m$ with $\gamma_1=\gamma_m=0$ such that $f_{d-d'+1}(\bm{\gamma},\bm{\tau})=1$ for each $(\tau_{d-3},\ldots,\tau_{d+1})\in\FF_2^5$. This completes the proof.
\end{proof}

\begin{lemma}\label{lm:number of solutoins}
      Suppose that $\eqref{enu:C1}$ and $\eqref{enu:C2}$ hold, let $\tau\in a_1H_{2,d+1}$ be fixed, and let $f_t$ be as in~\eqref{eq:regular_f_i}. Then for each $\{0,1\}$-sequence $(\gamma_t)_{t=-d'+1}^{d-d'}$, there are at most $2^{d'-2}$  common solutions $\bm{\alpha}\in \FF_2^m$ to the equations 
      \begin{equation}\label{eq:system}
            f_t(\bm{\alpha})=\gamma_t\ \text{ for } \  t\in \{-d'+1,\ldots,d-d'\}.
      \end{equation}
\end{lemma}

\begin{proof}
      Let $u$ be as defined in~\eqref{eq:u} and $\bm{\alpha}=(\alpha_1,\ldots,\alpha_m)$ be a common solution to~\eqref{eq:system}.
      Then, by Lemma~\ref{lm:expressions of f_t}\eqref{enu:f_t1}, we obtain that for each $t\in \{1,\ldots, m-u+1\}$,
      \[
            \alpha_{t+u-1}\in f_t(\bm{\alpha})+ \FF_2[\alpha_{d+1},\ldots,\alpha_{t+u-2}]=\gamma_t+\FF_2[\alpha_{d+1},\ldots,\alpha_{t+u-2}]=\FF_2[\alpha_{d+1},\ldots,\alpha_{t+u-2}].
      \]
      Similarly, we deduce from Lemma~\ref{lm:expressions of f_t}\eqref{enu:f_t2}\eqref{enu:f_t3} that 
      \begin{align*}
            & \alpha_{t-m+u}\in \FF_2[\alpha_1,\ldots,\alpha_{t-m+u-1},\alpha_{d+1},\ldots,\alpha_m] \ \text{ for each }\ t\in \{m-u+1,\ldots,d-d'+1\},\\
            &\alpha_{d+t}\in \FF_2[\alpha_1,\ldots,\alpha_{d'},\alpha_{d+1+t},\ldots,\alpha_m] \ \ \text{ for each }\ t\in \{-d'+1,\ldots,0\}.
      \end{align*}
      Changing the index, we obtain 
      \begin{align}
            &\alpha_{t}\in \FF_2[\alpha_{d+1},\ldots,\alpha_{t-1}]\ \text{ for each }\ t\in\{u,\ldots,m\},\label{eq:um}\\
            & \alpha_{t}\in \FF_2[\alpha_1,\ldots,\alpha_{t-1},\alpha_{d+1},\ldots,\alpha_m] \ \text{ for each }\ t\in \{1,\ldots,u-2d'+1\},\label{eq:1u-2d'+1}\\
            &\alpha_{t}\in \FF_2[\alpha_1,\ldots,\alpha_{d'},\alpha_{t+1},\ldots,\alpha_m] \ \ \text{ for each }\ t\in \{d-d'+1,\ldots,d\}.\label{eq:d-d'+1d}
      \end{align}

      
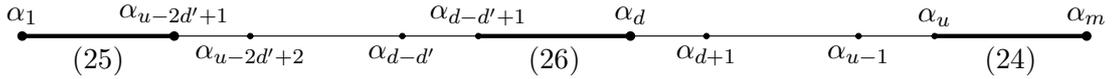
\vspace{5pt}
      \begin{center}
      
      \begin{tikzpicture}
            \filldraw [ultra thick]
            (0,0) circle (1pt) node [above] {$\alpha_1$} -- 
            (2,0) circle (1pt) node [above] {$\alpha_{u-2d'+1}$};
            \draw (1,0) node [below] {\eqref{eq:1u-2d'+1}};
            \filldraw 
            (2,0) -- 
            (3,0) circle (1pt) node [below] {$\alpha_{u-2d'+2}$} -- 
            (5,0) circle (1pt) node [below] {$\alpha_{d-d'}$} -- 
            (6,0) circle (1pt) node [above] {$\alpha_{d-d'+1}$} ;
            \filldraw [ultra thick]
            (6,0) -- 
            (8,0) circle (1pt) node [above] {$\alpha_d$} ;
            \draw (7,0) node [below] {\eqref{eq:d-d'+1d}};
            \filldraw 
            (8,0)-- 
            (9,0) circle (1pt) node [below] {$\alpha_{d+1}$} -- 
            (11,0) circle (1pt) node [below] {$\alpha_{u-1}$} -- 
            (12,0) circle (1pt) node [above] {$\alpha_u$} ;
            \filldraw [ultra thick]
            (12,0) --
            (14,0)circle (1pt) node [above] {$\alpha_m$};
            \draw (13,0) node [below] {\eqref{eq:um}};
      \end{tikzpicture}
      \end{center}
      Fix arbitrary values for $\alpha_{u-2d'+2},\ldots,\alpha_{d-d'},\alpha_{d+1},\ldots,\alpha_{u-1}$. Then we see from~\eqref{eq:um} that the values of $\alpha_u,\ldots,\alpha_m$ are uniquely determined.  Now that the values of $\alpha_{d+1},\ldots,\alpha_{m}$ are fixed, we derive from~\eqref{eq:1u-2d'+1} that the values of $\alpha_1,\ldots,\alpha_{u-2d'+1}$ are also uniquely determined. Similarly, using~\eqref{eq:d-d'+1d} and the fixed values of $\alpha_1,\ldots,\alpha_{d-d'},\alpha_{d+1},\ldots,\alpha_m$, we can uniquely determine the values of $\alpha_{d-d'+1},\ldots,\alpha_d$. Hence the number of solutions to~\eqref{eq:system} is no more than the number of choices of $\alpha_{u-2d'+2},\ldots,\alpha_{d-d'},\alpha_{d+1},\ldots,\alpha_{u-1}$, which is $2^{d'-2}$.
\end{proof}

\section{Primitivity}\label{sec:primitive}

Throughout this section, all notation follows the definitions provided in Section~\ref{sec:preliminary}. Let $H$ be an arbitrary concentric group. Our goal is to prove the following proposition.

\begin{proposition}\label{prop:primitive2}
      Suppose that $\eqref{enu:C1}$ holds. Then there exists $\tau_{d+1}\in\FF_2$ such that the group $\langle R(H),x_\tau\rangle$ is primitive on $H$ for each $(\tau_2,\ldots,\tau_d)\in\FF_2^{d-1}$.
\end{proposition}

We need the following lemma to prove Proposition~\ref{prop:primitive2}.

\begin{lemma}\label{lm:Nsubg}
      Let $N$ be a block of imprimitivity of $\langle R(H),x_\tau\rangle$ acting on $H$ such that $1\in N$ and $N\neq H$. Then the following hold:
      \begin{enumerate}[\rm(a)]
            \item \label{enu:lm3.1 a} $N$ is a subgroup of $H$ such that $N^{x_\tau}=N$;
            \item \label{enu:lm3.1 b} $N\cap\{a_1,a_2,\ldots, a_m\}= \emptyset$;
            \item \label{enu:lm3.1 d} If $[a_1,a_{d+1}]\in \{a_1,\ldots, a_m\}$, then $N\cap Z(H)=1$.
      \end{enumerate}
\end{lemma}

\begin{proof}
      \rm(a) Note for each $h_1,h_2\in N$ that
      \[
            \{1,h_1,h_2\}^{R(h_2^{-1})}\cap \{1,h_1,h_2\}=\{h_2^{-1}, h_1h_2^{-1},1\}\cap \{1,h_1,h_2\}\neq \emptyset.
      \]
      Since $N$ is a block of $\langle R(H),x_\tau\rangle$, we have $N^{R(h_2^{-1})}=N$, and thus $h_1h_2^{-1}\in N$, which implies that $N$ is a subgroup of $H$. Moreover, since $1\in N^{x_\tau}\cap N$, we conclude that $N^{x_\tau}=N$.
      
      \rm(b) Suppose $N\cap\{a_1,a_2,\ldots, a_m\}\neq \emptyset$. By the definition of $x_\tau$, we have 
      \[a_m=a_{m-1}^{x_\tau}=\dots=a_1^{x_\tau^{m-1}},\]
      which, by part~\eqref{enu:lm3.1 a}, implies that $\{a_1,a_2,\ldots,a_m\}\subseteq N$. Since $N$ is a subgroup of $H$, we conclude that $H=\langle a_1,a_2,\ldots,a_{m}\rangle\subseteq N$, and thus $N=H$, a contradiction.
            
      \rm(c) Suppose that $[a_1,a_{d+1}]\in \{a_1,\ldots, a_m\}$ and $N\cap Z(H)>1$. Then $N$ contains a non-identity element of $Z(H)=H_{d'+1,d}$, say, $a_{k_1}\dots a_{k_s}$, where $d'+1\leq k_1<\dots< k_s\leq d$. Then we derive from part~\eqref{enu:lm3.1 a} that
      \[
            a_1a_{k_2-k_1+1}\dots a_{k_s-k_1+1}=(a_{k_1}a_{k_2}\dots a_{k_s})^{x_\tau^{-k_1+1}}\in N^{x_\tau^{-k_1+1}}=N,
      \]
      \[
            a_{k_1-k_s+d+1}\dots a_{k_{s-1}-k_s+d+1}a_{d+1}=(a_{k_1}\dots a_{k_{s-1}}a_{k_s})^{x_\tau^{-k_s+d+1}}\in N^{x_\tau^{-k_s+d+1}}=N.
      \]
      Noting from Definition~\ref{def:rel} that 
      \begin{align*}
            &[a_{k_2-k_1+1}\dots a_{k_s-k_1+1},a_{k_1-k_s+d+1}\dots a_{k_{s-1}-k_s+d+1}a_{d+1}]=1,\\
            &[a_1a_{k_2-k_1+1}\dots a_{k_s-k_1+1},a_{k_1-k_s+d+1}\dots a_{k_{s-1}-k_s+d+1}]=1,
      \end{align*}
      we obtain 
      \[
            [a_1a_{k_2-k_1+1}\dots a_{k_s-k_1+1},a_{k_1-k_s+d+1}\dots a_{k_{s-1}-k_s+d+1}a_{d+1}]=[a_1,a_{d+1}].
      \]
      It follows that $[a_1,a_{d+1}]\in N$, and so $ N\cap \{a_1,\ldots, a_m\}\neq \emptyset$, contradicting part~\eqref{enu:lm3.1 b}.\qedhere
\end{proof}

\begin{proof}[Proof of Proposition~\ref{prop:primitive2}]
      Suppose for a contradiction that $N$ is a nontrivial block of $\langle R(H),x_\tau\rangle$ acting on $H$ such that $1\in N$. Then Lemma~\ref{lm:Nsubg}\eqref{enu:lm3.1 a} asserts that $N$ is a subgroup of $H$ stabilized by $x_\tau$. By Lemma~\ref{lm:Nsubg}\eqref{enu:lm3.1 d}, since the condition~\eqref{enu:C1} implies that $[a_1,a_{d+1}]\in \{a_1,\ldots,a_m\}$, we only need to show that $N\cap Z(H)>1$.
      
      Take a non-identity element in $N$, say, $a_{k_1}a_{k_2}\dots a_{k_s}$, where $1\leq k_1<\dots<k_s\leq m$. Then
      \[
            a_{k_1-k_s+m}\dots a_{k_{s-1}-k_s+m}a_m=(a_{k_1}\dots a_{k_{s-1}}a_{k_s})^{x_\tau^{-k_s+m}}\in N^{x_\tau^{-k_s+m}}=N.
      \]
      Let $h=a_{k_1-k_s+m}\dots a_{k_{s-1}-k_s+m}$. Then $h\in H_{1,m-1}$ and $ha_m=a_{k_1-k_s+m}\dots a_{k_{s-1}-k_s+m}a_m\in N$. It follows that $((ha_m)^{x_\tau})^2\in N^{x_\tau} N^{x_\tau}=N$. As Lemma~\ref{lm:concentric} indicates that $((ha_m)^{x_\tau})^2\in Z(H)$, we are left to choose a $\tau_{d+1}\in\FF_2$ such that $((ha_m)^{x_\tau})^2$ is non-identity.
      
      It is straightforward to calculate that
      \begin{align*}
            ((ha_m)^{x_\tau})^2
            =((a_mh[h,a_m])^{x_\tau})^2
            &=(\tau h^\varphi [h,a_m]^\varphi)^2\\
            &=\tau^2 (h^\varphi)^2 [\tau,h^\varphi][\tau,[h,a_m]^\varphi]
            =\tau^2 (h^2)^\varphi [\tau,h^\varphi][\tau,[h,a_m]^\varphi].
      \end{align*}
      Now we calculate the $(d'+1)$-th coordinate of $((ha_m)^{x_\tau})^2$.
      Since $\tau=a_1a_2^{\tau_2}\dots a_{d+1}^{\tau_{d+1}}$, we have $\tau^2=[a_1,a_{d+1}]^{\tau_{d+1}}=a_{d'+1}^{\tau_{d+1}}$. By Lemma~\ref{lm:concentric}, we derive that $(h^2)^\varphi\in Z(H)^\varphi=H_{d'+2,d+1}$ and $(h^2)^\varphi=(h^\varphi)^2\in Z(H)=H_{d'+1,d}$, which implies that 
      \begin{equation}\label{eq:(h^2)^varphi}
            (h^2)^\varphi\in H_{d'+2,d+1}\cap H_{d'+1,d}=H_{d'+2,d}.
      \end{equation}
      Moreover, since $h^\varphi\in H_{2,m}$,
      \begin{align*}
            [\tau,h^\varphi]H_{d'+2,d}
            =[a_1a_2^{\tau_2}\dots a_{d+1}^{\tau_{d+1}},h^\varphi]H_{d'+2,d}
            =[a_1,h^\varphi]H_{d'+2,d}
            =a_{d'+1}^{\gamma}H_{d'+2,d}
      \end{align*}
      for some $\gamma\in \FF_2$. Let $\omega\in \FF_2$ be the $d$-th coordinate of $[h,a_m]$. Then $[\tau,[h,a_m]^\varphi]=[a_1,a_{d+1}]^\omega$. Note that
      \[
            (ha_m)^2=ha_mha_m=h^2a_m^2[h,a_m]=h^2[h,a_m].
      \]
      Then since~\eqref{eq:(h^2)^varphi} implies that $h^2\in H_{d'+1,d-1}$, the $d$-th coordinate of $(ha_m)^2$ is the $d$-th coordinate of $[h,a_m]$, which is $\omega$. If $\omega=1$, then $(ha_m)^2$ is a non-identity element in $N\cap Z(H)$, as desired. Now assume that $\omega=0$. Then $[\tau,[h,a_m]^\varphi]=1$, and the $(d'+1)$-th coordinate of $((ha_m)^{x_\tau})^2=\tau^2(h^2)^\varphi[\tau,h^\varphi][\tau,[h,a_m]^\varphi]$ is $\tau_{d+1}+\gamma$. Take $\tau_{d+1}=\gamma+1$. Then the $(d'+1)$-th coordinate of $((ha_m)^{x_\tau})^2$ is $1$, which implies that $((ha_m)^{x_\tau})^2$ is non-identity. This completes the proof.
\end{proof}

\section{Affineness}\label{sec:Affine}

In this section, we adhere to the notation established in Section~\ref{sec:preliminary}. Let $H$ be an arbitrary concentric group. We aim to prove the following proposition.

\begin{proposition}\label{prop:affine}
      Suppose that $3d>2m$, and that $\eqref{enu:C1}$ holds. Then for each $(\tau_2,\ldots,\tau_{d-4},\tau_{d-1},\tau_d,\tau_{d+1})\in\FF_2^{d-2}$ such that 
      \begin{equation}\label{eq:sum}
            \tau_{d'}=1+
            \sum\limits_{k=1}^{d'-1} \Big( \sum\limits_{i=k}^{d'}\sum\limits_{\ell=i}^{d-d'} (i+k+1)\varepsilon_{m-i,\ell-i}\Big)\tau_k,
      \end{equation}
      there exists $(\tau_{d-3},\tau_{d-2})\in\FF_2^2$ such that the group $\langle R(H),x_\tau\rangle$ is not contained in any affine group on $H$.
\end{proposition}

\begin{remark}
      The condition $3d>2m$ guarantees that $d-d'=2d-m>m/3\geq 3$, and thus, $d-3>d'$.
\end{remark}

The approach to proving Proposition~\ref{prop:affine} is based on the following observation.

\begin{lemma}\label{lm:affine}
      Suppose that $\langle R(H),x_\tau\rangle$ contains an element with the number of fixed points odd and greater than $1$. Then $\langle R(H),x_\tau\rangle$ is not a subgroup of any affine group on $H$.
\end{lemma}

\begin{proof}
      Suppose for a contradiction that $\langle R(H),x_\tau\rangle$ is a subgroup of some affine group. Since the degree of $\langle R(H),x_\tau\rangle$ is $2^m$, we have $\langle R(H),x_\tau\rangle\leq \AGL_m(2)$. Note that $\GL_m(2)$ is a point stabilizer in $\AGL_m(2)$ and the set of fixed points of any element in $\GL_m(2)$ forms a vector space over $\FF_2$. Thus, the number of fixed points of any element in $\AGL_m(2)$ is either $0$ or a power of $2$. This contradicts the hypothesis that $\langle R(H),x_\tau\rangle$ contains an element with the number of fixed points odd and greater than $1$.
\end{proof}

We aim to find $\tau\in a_1H_{2,m}$ such that the number of fixed points of $x_\tau^2$ is odd and greater than $1$. To achieve this, we give the following two lemmas.

\begin{lemma}\label{lm:x fixed points}
      The identity $1$ is the unique fixed point of $x_\tau$ if
      \begin{equation}\label{eq:E1}
            \sum_{i=1}^{d+1} \tau_i+\sum_{\ell=1}^{d-d'}\sum_{i=1}^{\min\{d',\ell\}}\sum_{k=1}^i\tau_k\varepsilon_{m-i,\ell-i}=0.\tag{E1}
      \end{equation}
\end{lemma}

\begin{proof}
      By Lemma~\ref{lm:x coordinate form}, a point $\bm{\alpha}=(\alpha_1,\ldots,\alpha_m)$ is fixed by $x_\tau$ if and only if 
      \begin{equation}\label{eq:5.1x}
            \bm{\alpha}=\bm{\alpha}^\phi+\alpha_m\Big(\bm{e}_1+\bm{\tau}+\sum_{\ell=d'+2}^{d+1}\lambda_{\ell-1}(\bm{e}_m,\bm{\alpha})\bm{e}_{\ell}\Big).
      \end{equation}
      Thus, $\bm{0}$ is the unique fixed point $\bm{\alpha}$ of $x_\tau$ such that $\alpha_m=0$. 
      Now let $\alpha_m=1$. Then we obtain a system of equations from~\eqref{eq:5.1x}:
      \begin{equation}\label{eq:x system of equations}
            \begin{cases}
                  \alpha_1&=\ \ \tau_1 \\
                  \alpha_\ell&=\ \ \tau_\ell+\alpha_{\ell-1}\ \ \text{for}\ \ \ell\in\{2,\ldots,d'+1\}\\
                  \alpha_\ell&=\ \ \tau_\ell+\alpha_{\ell-1}+\lambda_{\ell-1}(\bm{e}_m,\bm{\alpha})\ \ \text{for}\ \ \ell\in \{d'+2,\ldots,d+1\}\\
                  \alpha_\ell&=\ \ \alpha_{\ell-1}\ \ \text{for}\ \ \ell\in\{d+2,\ldots,m\}.
            \end{cases}
      \end{equation}
      Summing up all the equations in~\eqref{eq:x system of equations} gives 
      \begin{equation}\label{eq:E1_sum_all}
            \alpha_m=\sum_{i=1}^{d+1} \tau_i+\sum_{\ell=d'+2}^{d+1}\lambda_{\ell-1}(\bm{e}_m,\bm{\alpha}),
      \end{equation}
      while summing up the first $t$ equations, where $t\in\{1,\ldots,d'+1\}$ gives 
      \[
            \alpha_t=\sum_{k=1}^t \tau_k.
      \]
      Taking this into~\eqref{eq:E1_sum_all}, we obtain 
      \begin{align*}
            1=\alpha_m=\sum_{i=1}^{d+1} \tau_i+\sum_{\ell
            =d'+2}^{d+1}\lambda_{\ell-1}(\bm{e}_m,\bm{\alpha})
            &=\sum_{i=1}^{d+1} \tau_i+\sum_{\ell=1}^{d-d'}\lambda_{\ell+d'}(\bm{e}_m,\bm{\alpha})\\
            &=\sum_{i=1}^{d+1} \tau_i+\sum_{\ell=1}^{d-d'}\sum_{i=1}^{\min\{d',\ell\}}\sum_{k=1}^i\tau_k\varepsilon_{m-i,\ell-i}
      \end{align*}
      by~\eqref{eq:e_m,alpha}.
      Hence, if~\eqref{eq:E1} is satisfied, then there is no fixed point $\bm{\alpha}$ of $x_\tau$ such that $\alpha_m=1$. This completes the proof.
\end{proof}

\begin{lemma}\label{lm:x^2 fixed points}
      The permutation $x_\tau^2$ has more than one fixed points if the following equalities hold:
      \begin{align}
            &\sum\limits_{i=1}^{d+1} (m+i)\tau_i+\sum\limits_{\ell=1}^{d-d'} (\ell+d+1) \sum\limits_{i=1}^{\min\{d',\ell\}}\sum\limits_{k=1}^{i-1} (i+k)\tau_k\varepsilon_{m-i,\ell-i}=1,\tag{E2}\label{eq:E2}\\
            &\sum\limits_{i=1}^{d+1} (m+i+1)\tau_i+\sum\limits_{\ell=1}^{d-d'} (\ell+d) \sum\limits_{i=1}^{\min\{d',\ell\}}\sum\limits_{k=1}^{i-1} (i+k)\tau_k\varepsilon_{m-i,\ell-i}=0.\tag{E3}\label{eq:E3}
      \end{align}
\end{lemma}

\begin{proof}
      Since $1$ is a fixed point of $x_\tau^2$, it suffices to prove the existence of another fixed point of $x_\tau^2$. To achieve this, we will show that the equation 
      \begin{equation}\label{eq:x^2 fixed point}
            (\alpha_1,\ldots,\alpha_{m-2},1,0)=(\alpha_1,\ldots,\alpha_{m-2},1,0)^{x_\tau^2}
      \end{equation}
      on $(\alpha_1,\ldots,\alpha_{m-2})$ has a solution. Let $\bm{\alpha}=(\alpha_1,\ldots,\alpha_{m-2},1,0)$. Then by Lemma~\ref{lm:x coordinate form},
      \[
            \bm{\alpha}^{x_\tau^2}
            =(\bm{\alpha}^\phi)^{x_\tau}
            =\bm{\alpha}^{\phi^2}+\bm{e}_1+\bm{\tau}+\sum_{\ell=d'+2}^{d+1}\lambda_{\ell-1}(\bm{e}_m,\bm{\alpha}^\phi)\bm{e}_{\ell}.
      \]
      Thus, equation~\eqref{eq:x^2 fixed point} is equivalent to the system 
      \[
            \begin{cases}
                  \alpha_1& \hspace{-10pt}=\ \ \tau_1\\
                  \alpha_2&\hspace{-10pt}=\ \ \tau_2\\
                  \alpha_\ell&\hspace{-10pt}=\ \ \tau_\ell+\alpha_{\ell-2}\ \ \text{for}\ \ \ell\in\{3,\ldots,d'+1\}\\
                  \alpha_\ell&\hspace{-10pt}=\ \ \tau_\ell+\alpha_{\ell-2}
                  +\lambda_{\ell-1}(\bm{e}_m,\bm{\alpha}^\phi)\ \ \text{for}\ \ \ell\in\{d'+2,\ldots,d+1\}\\
                  \alpha_\ell&\hspace{-10pt}=\ \ \alpha_{\ell-2}\ \ \text{for}\ \ \ell\in\{d+2,\ldots,m\}\\
                  \alpha_{m-1}&\hspace{-10pt}=\ \ 1\\
                  \alpha_m&\hspace{-10pt}=\ \ 0.
            \end{cases}
      \]
      Summing up the $1$-st, $3$-rd,$\ldots,(2s-1)$-th equations and the $2$-nd, $4$-th,$\ldots,(2s)$-th equations respectively, for each $s \in \{1,\ldots,\lceil m/2\rceil\}$, we conclude that the above system is equivalent to
      \[
            \begin{cases}
                  \alpha_t&\hspace{-10pt}=\ \ \sum\limits_{i=1}^t (t+i+1)\tau_i\ \ \text{for}\ \ t\in\{1,\ldots, d'+1\}\\
                  \alpha_t&\hspace{-10pt}=\ \ \sum\limits_{i=1}^{\min\{t,d+1\}} (t+i+1)\tau_i+\sum\limits_{\ell=d'+2}^{\min\{t,d+1\}}(t+\ell+1)\lambda_{\ell-1}(\bm{e}_m,\bm{\alpha}^\phi)\ \ \text{for}\ \ t\in\{d'+2,\ldots, m\}\\
                  \alpha_{m-1}&\hspace{-10pt}=\ \ 1\\
                  \alpha_m&\hspace{-10pt}=\ \ 0.
            \end{cases}
      \]
      Substituting the first $d'+1$ equations into all the other equations and~\eqref{eq:e_m,alpha} into $\lambda_{\ell-1}(\bm{e}_m,\bm{\alpha}^\varphi)$, we obtain another equivalent system of equations:
      \[
            \begin{cases}
                  \alpha_t&\hspace{-10pt}=\ \ \sum\limits_{i=1}^t (t+i+1)\tau_i\ \ \text{for}\ \ t\in\{1,\ldots, d'+1\}\\
                  \alpha_t&\hspace{-10pt}=\ \ \sum\limits_{i=1}^t (t+i+1)\tau_i+\sum\limits_{\ell=d'+2}^{\min\{t,d+1\}}(t+\ell+1) \sum\limits_{i=1}^{\min\{d',\ell-d'-1\}} \sum\limits_{k=1}^{i-1}(i+k)\tau_k\varepsilon_{m-i,\ell-d'-1-i}\\
                  &\hspace{-10pt}\quad\ \ \text{for}\ \ t\in\{d'+2,\ldots, m\}\\
                  \alpha_{m-1}&\hspace{-10pt}=\ \ 1\\
                  \alpha_m&\hspace{-10pt}=\ \ 0.
            \end{cases}
      \]    
      It follows that the system of equations on $(\alpha_1,\ldots,\alpha_m)$ has a unique solution if and only if the following equations hold:
      \begin{align*}
            1=&\;\sum\limits_{i=1}^{m-1} (m+i)\tau_i+\sum\limits_{\ell=d'+2}^{d+1} (m+\ell) \sum\limits_{i=1}^{\min\{d',\ell-d'-1\}}\sum\limits_{k=1}^{i-1} (i+k)\tau_k\varepsilon_{m-i,\ell-d'-1-i},\\
            0=&\;\sum\limits_{i=1}^{m} (m+i+1)\tau_i+\sum\limits_{\ell=d'+2}^{d+1} (m+\ell+1) \sum\limits_{i=1}^{\min\{d',\ell-d'-1\}}\sum\limits_{k=1}^{i-1} (i+k)\tau_k\varepsilon_{m-i,\ell-d'-1-i}.
      \end{align*}
      These two equations are equivalent to~\eqref{eq:E2} and~\eqref{eq:E3} respectively.
\end{proof}

Note that the set of fixed points of $x_\tau^2$ is exactly the union of all the orbits of $x_\tau$ of length at most $2$. We infer from Lemmas~\ref{lm:affine},~\ref{lm:x fixed points} and~\ref{lm:x^2 fixed points} that the proof of Proposition~\ref{prop:affine} now centers on identifying the common solutions of $\eqref{eq:E1}$, $\eqref{eq:E2}$ and $\eqref{eq:E3}$. 

\begin{proof}[Proof of Proposition~\ref{prop:affine}]
      If $\bm{\tau}=(\tau_1,\ldots,\tau_m)$ satisfies $\eqref{eq:E1}$, $\eqref{eq:E2}$ and $\eqref{eq:E3}$, then it follows from Lemmas~\ref{lm:x fixed points} and~\ref{lm:x^2 fixed points} that the number of fixed points of $x_\tau^2$ is odd and greater than $1$, and we conclude from Lemma~\ref{lm:affine} that $\langle R(H),x_\tau\rangle$ is not a subgroup of any affine group on $H$. Hence, to prove the proposition, it suffices to show that for each $(\tau_2,\ldots,\tau_{d-4},\tau_{d-1},\tau_d,\tau_{d+1})\in\FF_2^{d-2}$ such that~\eqref{eq:sum} holds, there exists $(\tau_{d-3},\tau_{d-2})\in\FF_2^2$ such that $(\tau_2,\ldots,\tau_{d+1})$ is a common solution to~\eqref{eq:E1},~\eqref{eq:E2} and~\eqref{eq:E3}. 
      
      Let $\eqref{eq:E1}+\eqref{eq:E2}+\eqref{eq:E3}$ be the equation obtained by summing up~\eqref{eq:E1},~\eqref{eq:E2} and~\eqref{eq:E3}. Then it is straightforward to calculate that $\eqref{eq:E1}+\eqref{eq:E2}+\eqref{eq:E3}$ is equal to
      \[
            \sum\limits_{\ell=1}^{d-d'} \sum\limits_{i=1}^{\min\{d',\ell\}}\sum\limits_{k=1}^{i} (i+k+1)\tau_k\varepsilon_{m-i,\ell-i}=1,
      \]
      which can be written as
      \begin{equation}\label{eq:E1-E2-E3}
            \sum\limits_{k=1}^{d'} \Big( \sum\limits_{i=k}^{d'}\sum\limits_{\ell=i}^{d-d'} (i+k+1)\varepsilon_{m-i,\ell-i}\Big)\tau_k=1.
      \end{equation}
      We then deduce from~\eqref{enu:C1} that~\eqref{eq:E1-E2-E3} that
      \[
            \tau_{d'}+
            \sum\limits_{k=1}^{d'-1} \Big( \sum\limits_{i=k}^{d'}\sum\limits_{\ell=i}^{d-d'} (i+k+1)\varepsilon_{m-i,\ell-i}\Big)\tau_k=1,
      \]
      which is equivalent to~\eqref{eq:sum}.
      Hence, for each $ (\tau_2,\ldots,\tau_{d-4},\tau_{d-1},\tau_d,\tau_{d+1})\in\FF_2^{d-2}$ such that~\eqref{eq:sum} holds, $\eqref{eq:E1}+\eqref{eq:E2}+\eqref{eq:E3}$ holds for all $(\tau_{d-3},\tau_{d-2})\in\FF_2^2$. 
      
      Fix an arbitrary $(\tau_2,\ldots,\tau_{d-4},\tau_{d-1},\tau_d,\tau_{d+1})\in\FF_2^{d-2}$ such that ~\eqref{eq:sum} holds. Then,~\eqref{eq:E2} and~\eqref{eq:E3} can be written as
      \[
            (m+d-3)\tau_{d-3}+(m+d-2)\tau_{d-2}=\beta\ \ \text{and}\ \ (m+d-2)\tau_{d-3}+(m+d-1)\tau_{d-2}=\gamma
      \]
      respectively, for some $\beta,\gamma\in\FF_2$. Noting that
      \[
            \begin{pmatrix}
                  m+d-3 & m+d-2\\
                  m+d-2 & m+d-1
            \end{pmatrix}\in
            \left\{
            \begin{pmatrix}
                  1 & 0\\
                  0 & 1
            \end{pmatrix},
            \begin{pmatrix}
                  0 & 1\\
                  1 & 0
            \end{pmatrix}
            \right\},
      \]
      we assert that there exists $(\tau_{d-3},\tau_{d-2})\in\FF_2^2$ such that~\eqref{eq:E2} and~\eqref{eq:E3} hold simultaneously. Since $\eqref{eq:E1}+\eqref{eq:E2}+\eqref{eq:E3}$ holds for all $(\tau_{d-3},\tau_{d-2})\in\FF_2^2$, we conclude that there exists $(\tau_{d-3},\tau_{d-2})\in\FF_2^2$ such that $(\tau_2,\ldots,\tau_{d+1})$ is a common solution to~\eqref{eq:E1},~\eqref{eq:E2} and~\eqref{eq:E3},  as desired.
\end{proof}

\section{Product Action}\label{sec:PA}

In this section, we use the notation defined in Section~\ref{sec:preliminary}. Let $H$ be an arbitrary concentric group. Our objective is to prove the following proposition.

\begin{proposition}\label{prop:exwr}
      Suppose that $3d>2m$, and that $\eqref{enu:C1}$ and $\eqref{enu:C2}$ hold. Then for each $(\tau_2,\ldots,\tau_{d-4},\tau_{d+1})\in\FF_2^{d-4}$, there exists $(\tau_{d-1},\tau_d)\in\FF_2^2$ such that for any $(\tau_{d-3},\tau_{d-2})\in\FF_2^2$, the group $\langle R(H),x_\tau\rangle$ is not a subgroup of $\Sym(\Delta)\wr S_k$ for any set $\Delta$ and positive integer $k$ with $|\Delta|=2^{m/k}\geq 5$.
\end{proposition}

We begin by proving two elementary lemmas regarding involutions in wreath products. For a permutation $\pi$ on a set $\Omega$, denote by $\Fix_\Omega(\pi)$ the set of fixed points of $\pi$ on $\Omega$, and by $\fpr_\Omega(\pi):=|\Fix_\Omega(\pi)|/|\Omega|$ the fixed point ratio of $\pi$ on $\Omega$. 

\begin{lemma}\label{lm:fixed points}
      Let $\Delta$ be a finite set, $k\geq 2$ be an integer, and let $g=(g_1,\ldots,g_k)\pi\in \Sym(\Delta)\wr S_k$ with $(g_1,\ldots,g_k)\in\Sym(\Delta)^k$ and $\pi\in S_k$. If $g$ is an involution, then $|\mathrm{Supp}(\pi)|$ is even and the number of fixed points of $g$ is
      \[ 
            |\Delta|^{\frac{|\mathrm{Supp}(\pi)|}{2}}\prod_{i\notin \mathrm{Supp}(\pi)} |\Fix_{\Delta}(g_i)|.
      \]
\end{lemma}

\begin{proof}
      Suppose that $g$ is an involution and let $t=|\mathrm{Supp}(\pi)|$. We first note that $\pi^2=1$ by considering the nature quotient from $\Sym(\Delta)\wr S_k$ to $S_k$. Thus, $t$ is even. Since the conclusion of the lemma is obvious when $\pi=1$, we assume $\pi\neq 1$. Then $\pi$ can be written as a product of disjoint transpositions $(i_1,j_1)\cdots(i_{t/2},j_{t/2})$. Since $\pi^2=1$ and $g^2=1$, it is straightforward to calculate that
      \begin{equation}\label{eq:wr1}
            (\beta_1,\ldots,\beta_k)^{g}=(\beta_{1^{\pi}}^{g_{1^{\pi}}},\ldots,\beta_{k^\pi}^{g_{k^{\pi}}})
      \end{equation}
      and 
      \[
            (\beta_1,\ldots,\beta_k)=(\beta_1,\ldots,\beta_k)^{g^2}=(\beta_{1^{\pi}}^{g_{1^{\pi}}},\ldots,\beta_{k^\pi}^{g_{k^{\pi}}})^g=(\beta_1^{g_1g_{1^{\pi}}},\ldots,\beta_k^{g_kg_{k^{\pi}}}).
      \]
      It follows that 
      \begin{equation}\label{eq:wr3}
            \beta_i^{g_ig_{i^\pi}}=\beta_i\quad\text{for each}\ \ i\in\{1,\ldots,k\}.
      \end{equation}
      Since $\pi=(i_1,j_1)\cdots(i_{t/2},j_{t/2})$, we derive from~\eqref{eq:wr1} that $g$ fixes $(\beta_1,\ldots, \beta_k)$ if and only if 
      \[
            \beta_{i_s}=\beta_{i_s^\pi}^{g_{i_s^\pi}}=\beta_{j_s}^{g_{j_s}} \quad\text{and}\quad \beta_{j_s}=\beta_{j_s^\pi}^{g_{j_s^\pi}}=\beta_{i_s}^{g_{i_s}} \quad \text{for each}\ \ s\in \{1,\ldots,  t/2\}
      \]
      and 
      \[
            \beta_i=\beta_{i^\pi}^{g_{i^\pi}}=\beta_{i}^{g_i}\quad \text{for each}\quad i\notin\{i_1,j_1,\ldots,i_{t/2},j_{t/2}\}=\mathrm{Supp}(\pi).
      \]
      This combined with~\eqref{eq:wr3} implies that $(\beta_1,\ldots,\beta_k)$ is a fixed point of $g$ if and only if 
      \[
            \beta_{j_s}=\beta_{i_s}^{g_{i_s}}\quad \text{for each}\ \ s\in \{1,\ldots,  t/2\}
      \]
      and 
      \[
            \beta_i\in \Fix_{\Delta}(g_i)\quad \text{for each}\ \ i\notin \mathrm{Supp}(\pi).
      \]
      Hence, the number of fixed points of $g$ is 
      \[
            |\Delta|^{\frac{t}{2}}\prod_{i\notin \mathrm{Supp}(\pi)} |\Fix_{\Delta}(g_i)|,
      \]
      as required.
\end{proof}

\begin{lemma}\label{lm:coordinate}
      Let $G\leq \Sym(\Delta)\wr S_k$ be a primitive wreath product acting on $\Omega=\Delta^k$, where $k\geq 2$ and $|\Delta|=2^\ell$ for some integer $\ell\geq 2$. Let $g=(g_1,\ldots,g_k)\pi\in G$ be an involution with $(g_1,\ldots,g_k)\in \Sym(\Delta)^k$ and $\pi\in S_k$. If $\fpr_\Omega(g)=1/2$, then $g=(1,\ldots,1,g_j,1,\ldots,1)$ for some $j\in \{1,\ldots, k\}$ such that $g_j$ is an involution with $\fpr_{\Delta}(g_j)=1/2$.
\end{lemma}

\begin{proof}
      Suppose that $\fpr_\Omega(g)=1/2$ and let $t=|\mathrm{Supp}(\pi)|$. Since $g$ is an involution, we see from Lemma~\ref{lm:fixed points} that $t$ is even and
      \[
            \frac{1}{2}= \fpr_\Omega(g)=\frac{|\Delta|^{\frac{t}{2}}\prod_{i\notin \mathrm{Supp}(\pi)} |\Fix_{\Delta}(g_i)|}{|\Delta|^k}\leq \frac{|\Delta|^{\frac{t}{2}}|\Delta|^{k-t}}{|\Delta|^{k}}= \frac{1}{|\Delta|^{\frac{t}{2}}}.
      \]
      If $t>0$, then $|\Delta|\leq 2$, a contradiction. Thus $\pi=1$, $g=(g_1,\ldots,g_k)$, and so
      \[
            \prod_{i=1}^k |\Fix_{\Delta}(g_i)|=|\Fix_\Omega(g)|=\fpr_\Omega(g)|\Omega|=\frac{1}{2}|\Delta|^k=2^{k\ell-1}.
      \]
      This implies that $|\Fix_\Delta(g_i)|$ is a power of $2$ for each $i\in\{1,\ldots,k\}$. Therefore, there exists a unique $j\in \{1,\ldots,k\}$ such that $|\Fix_{\Delta}(g_j)|=2^{\ell-1}$ and $|\Fix_{\Delta}(g_i)|=2^{\ell}=|\Delta|$ for all $i\neq j$. Thus, we conclude that $g=(1,\ldots,1,g_j,1,\ldots,1)$ with $g_j$ an involution and $\fpr_{\Delta}(g_j)=1/2$.
\end{proof}

Suppose that $\langle R(H),x_\tau\rangle$ is a primitive subgroup of $\Sym(\Delta)\wr S_k$ acting on $\Omega=\Delta^k$ with product action, where $|\Delta|=2^\ell\geq 5$ and $k\ell=m$. It follows  that $\langle R(H),x_\tau\rangle$ acting on $\Omega$ is permutation isomorphic to $\langle R(H),x_\tau\rangle$ acting on $\FF_2^m$. In other words, there exists a bijection $\psi: \FF_2^m \rightarrow \Omega$ 
such that $\bm{\alpha}^{\psi g}=\bm{\alpha}^{g\psi}$ for each $\bm{\alpha}\in \FF_2^m$ and each $g\in \langle R(H),x_\tau\rangle$. This implies that for each $g\in \langle R(H),x_\tau\rangle$,
      \[
            \bm{\alpha}\in\Fix_{\FF_2^m}(g)
            \quad\text{if and only if}\quad 
            \bm{\alpha}^\psi\in\Fix_\Omega(g).
      \]

      Note by Lemma~\ref{lm:y} that $y_\tau$ is an involution with fixed point ratio $1/2$. Hence, $y_\tau^{x_\tau^t}$ is an involution with fixed point ratio $1/2$ for each $t\in\Z$. Given $\tau\in a_1H_{2,m}$, we derive from Lemma~\ref{lm:coordinate} that for each $t\in\Z$, there exists a unique $j_t\in \{1,\ldots,k\}$ such that 
      \begin{equation}\label{eq:j_t}
            y_\tau^{x_\tau^t}=(1,\ldots,1,g_{j_t}^{(t)},1,\ldots,1)\in \Sym(\Delta)^k,
      \end{equation}
      for some $g_{j_t}^{(t)}\in\Sym(\Delta)$ with $\fpr_\Delta(g_{j_t}^{(t)})=1/2$, where $g_{j_t}^{(t)}$ appears at the $j_t$-th coordinate.

      \begin{lemma}\label{lm:module}
            Given $\tau\in a_1H_{2,m}$, and let $j_t$ be as in~\eqref{eq:j_t}. If $x_\tau=(x_1,\ldots,x_k)\pi$ for some $x_i\in \Sym(\Delta)$ and $\pi\in S_k$, then $j_t=j_0^{\pi^t}$ for each $t\in \Z$.
      \end{lemma}
      
      \begin{proof}
            Since $\tau$ is given, we abbreviate $x_\tau$ and $y_\tau$ to $x$ and $y$, respectively. Take an arbitrary $t\in\Z$ and $(\beta_1,\ldots,\beta_k)\in \Omega$. By a straightforward calculation, we obtain
            \begin{align*}
                  (\beta_1,\ldots,\beta_k)^{y^{x^{t+1}}}
                  &=(\beta_1,\ldots,\beta_k)^{x^{-1}(1,\ldots,1,g_{j_t}^{(t)},1,\ldots,1)x}\\
                  &=(\beta_1,\ldots,\beta_k)^{\pi^{-1}(x_1^{-1},\ldots,x_k^{-1})(1,\ldots,1,g_{j_t}^{(t)},1,\ldots,1)(x_1,\ldots,x_k)\pi}\\
                  &=(\beta_{1^\pi},\ldots,\beta_{k^{\pi}})^{(x_1^{-1},\ldots,x_k^{-1})(1,\ldots,1,g_{j_t}^{(t)},1,\ldots,1)(x_1,\ldots,x_k)\pi}\\
                  &=(\beta_{1^\pi},\ldots,\beta_{k^{\pi}})^{(1,\ldots,1,(g_{j_t}^{(t)})^{x_{j_t}},1,\ldots,1)\pi}\\
                  &=(\beta_{1^\pi},\ldots,\beta_{(j_t-1)^\pi},\beta_{j_t^\pi}^{(g_{j_t}^{(t)})^{x_{j_t}}},\beta_{(j_t+1)^\pi},\ldots,\beta_{k^{\pi}})^{\pi}\\
                  &=(\beta_1,\ldots,\beta_{j_t^\pi-1},\beta_{j_t^\pi}^{(g_{j_t}^{(t)})^{x_{j_t}}},\beta_{j_t^\pi+1},\ldots,\beta_k)\\
                  &=(\beta_1,\ldots,\beta_k)^{(1,\ldots,1,(g_{j_t}^{(t)})^{x_{j_t}},1,\ldots,1)},
            \end{align*}
            where $(g_{j_t}^{(t)})^{x_{j_t}}$ appears at the $j_t^\pi$-th coordinate. Thus, 
            \[
                  (1,\ldots,1,g_{j_{t+1}}^{(t+1)},1,\ldots,1)=y^{x^{t+1}}=(1,\ldots,1,(g_{j_t}^{(t)})^{x_{j_t}},1,\ldots,1).
            \]
            This implies that $j_{t+1}=j_t^\pi$ for each $t\in\Z$, and so the desired conclusion holds.
      \end{proof}

\begin{definition}\label{def:determine}
      Let $f: \FF_2^m \to \FF_2$ be a function. We define $E(f)$ to be the subset of $\{1,\ldots,k\}$ such that $j\in E(f)$ if and only if there exists a nonempty proper subset $\Delta_1$ of $\Delta$ such that 
      \[
            \{f(\bm{\alpha})\mid \bm{\alpha}^\psi\in \Delta^{j-1}\times\Delta_1\times\Delta^{k-j}\}=\{0\} \ \text{ and }\ 
            \{f(\bm{\alpha})\mid \bm{\alpha}^\psi\in \Delta^{j-1}\times(\Delta\setminus\Delta_1)\times\Delta^{k-j}\}=\{1\}.
      \]
\end{definition}

\begin{remark}
      Clearly, if $j\in E(f)$, then for each $\beta_j\in \Delta$,  
      \[
           |\{f(\bm{\alpha})\mid \bm{\alpha}^\psi\in \Delta^{j-1}\times\{\beta_j\}\times\Delta^{k-j}\}|=1.
      \]
      
\end{remark}

\begin{lemma}\label{lm:determine}
      Suppose that $\eqref{enu:C1}$ and $\eqref{enu:C2}$ hold. Given $\tau\in a_1H_{2,m}$, we regard each function $f_t$ defined in~\eqref{eq:regular_f_i} as a function from $\FF_2^m$ to $\FF_2$. Let $j_t$ be as in~\eqref{eq:j_t}. Then $j_t\in E(f_t)$  for each $t\in\{-d'+1,\ldots,d-d'+1\}$.
\end{lemma}

\begin{proof}
      Take an arbitrary $t\in\{-d'+1,\ldots,d-d'+1\}$.
      It follows from Corollary~\ref{cor:fpiff} that $f_t(\bm{\alpha})=0$ if and only if $\bm{\alpha}\in\Fix_{\FF_2^m}(y^{x^t})$. Notice by~\eqref{eq:j_t} that $(\beta_1,\ldots,\beta_k)\in \Fix_\Omega(y^{x^t})$ if and only if $\beta_{j_t}\in \Fix_\Delta(g_{j_t}^{(t)})$. We then deduce that $f_t(\bm{\alpha})=0$ if and only if $\beta_{j_t}\in \Fix_\Delta(g_{j_t}^{(t)})$. Therefore, 
      \begin{align*}
            &\{f_t(\bm{\alpha})\mid \bm{\alpha}^\psi\in \Delta^{j_t-1}\times\Fix_\Delta(g_{j_t}^{(t)})\times\Delta^{k-j_t}\}=\{0\},\\
            &\{f_t(\bm{\alpha})\mid \bm{\alpha}^\psi\in \Delta^{j_t-1}\times\big(\Delta\setminus\Fix_\Delta(g_{j_t}^{(t)})\big)\times\Delta^{k-j_t}\}=\{1\}.
      \end{align*}
      Since $\fpr_\Delta(g_{j_t}^{(t)})=1/2$, the set $\Fix_\Delta(g_{j_t}^{(t)})$ is a nonempty proper subset of $\Delta$. Then we  conclude from Definition~\ref{def:determine} that $j_t\in E(f_t)$.
\end{proof}

\begin{lemma}\label{lm:not determine}
      Suppose that $3d>2m$, and that $\eqref{enu:C1}$ and $\eqref{enu:C2}$ hold. Then for each $(\tau_2,\ldots,\tau_{d-4},\tau_{d+1})\in\FF_2^{d-4}$, there exists $(\tau_{d-1},\tau_d)\in\FF_2^2$ such that for each $(\tau_{d-3},\tau_{d-2})\in \FF_2^2$, the index $j_t$ defined in~\eqref{eq:j_t} satisfies $j_{-d'+1}=\dots=j_{d-d'+1}$.
\end{lemma}

\begin{proof}
      Fix an arbitrary $(\tau_2,\ldots,\tau_{d-4},\tau_{d+1})\in\FF_2^{d-4}$. By Lemma~\ref{lm:findGamma}, there exists $\bm{\gamma}\in \FF_2^m$ with $\gamma_1=\gamma_m=0$ such that $f_{d-d'+1}(\bm{\gamma},\bm{\tau})=1$ for all $(\tau_{d-3},\ldots,\tau_d)\in\FF_2^4$. Let us fix such a $\bm{\gamma}$. For convenience, write 
      \[
           z_\tau=x_\tau^{-(d-d'+1)}y_\tau x_\tau^{d-d'+1}.
      \]
      We first show that there exists $(\tau_{d-1},\tau_d)\in \FF_2^2$ such that  
      \begin{equation}\label{eq:plus1}
            f_0(\bm{\gamma}^{z_\tau},\bm{\tau})=f_0(\bm{\gamma},\bm{\tau})+1 \ \text{ and }\  
            f_{-1}(\bm{\gamma}^{z_\tau},\bm{\tau})=f_{-1}(\bm{\gamma},\bm{\tau})+1
      \end{equation} 
      for all $(\tau_{d-3},\tau_{d-2})\in\FF_2^2$.

      For all $(\tau_{d-3},\tau_{d-2},\tau_{d-1},\tau_d)\in \FF_2^4$, we have by Lemma~\ref{lm:f_0,f_1} and $\gamma_m=0$ that 
      \begin{align*}
            f_0(\bm{\gamma}^{z_\tau},\bm{\tau})&=f_0(\bm{\gamma},\bm{\tau})+\tau_d+\lambda_{d-1}(\bm{e}_m,\bm{\gamma}^{x_\tau^{-1}})+\tau_{d+1}\varepsilon_{d+1,0},\\
            f_{-1}(\bm{\gamma}^{z_\tau},\bm{\tau})
            &=f_{-1}(\bm{\gamma},\bm{\tau})+\tau_{d-1}+\lambda_{d-2}(\bm{e}_m,\bm{\gamma}^{x_\tau^{-1}})+\big(\tau_d+\lambda_{d-1}(\bm{e}_m,\bm{\gamma}^{x_\tau^{-1}})\big)\varepsilon_{d+1,0}+\tau_{d+1}\varepsilon_{d+2,0}.
      \end{align*}
      Note by $\gamma_1=0$ and Lemma~\ref{lm:x^-1 coordinate form} that $\lambda_i(\bm{e}_m,\bm{\gamma}^{x_\tau^{-1}})=\lambda_i(\bm{e}_m,\bm{\gamma}^{\phi^{-1}})$ for each $i\in \{d-2,d-1\}$. 
      Thus, $\lambda_{d-2}(\bm{e}_m,\bm{\gamma}^{x_\tau^{-1}})$ and $\lambda_{d-1}(\bm{e}_m,\bm{\gamma}^{x_\tau^{-1}})$ are fixed values as $\bm{\gamma}$ is fixed. Take
      \[
            \tau_d=1+\lambda_{d-1}(\bm{e}_m,\bm{\gamma}^{x_\tau^{-1}})+\tau_{d+1}\varepsilon_{d+1,0}.
      \]
      Then the value of $\tau_d$ is fixed, and $f_0(\bm{\gamma}^{z_\tau},\bm{\tau})=f_0(\bm{\gamma},\bm{\tau})+1$ for all $(\tau_{d-3},\tau_{d-2},\tau_{d-1})\in \FF_2^3$. Take
      \[
            \tau_{d-1}=1+\lambda_{d-2}(\bm{e}_m,\bm{\gamma}^{x_\tau^{-1}})+\big(\tau_d+\lambda_{d-1}(\bm{e}_m,\bm{\gamma}^{x_\tau^{-1}})\big)\varepsilon_{d+1,0}
            +\tau_{d+1}\varepsilon_{d+2,0}.
      \]
      Then the value of $\tau_{d-1}$ is fixed and $f_{-1}(\bm{\gamma}^{z_\tau},\bm{\tau})=f_{-1}(\bm{\gamma},\bm{\tau})+1$ for all $(\tau_{d-3},\tau_{d-2})\in \FF_2^2$. Hence, there exists $(\tau_{d-1},\tau_d)\in \FF_2^2$ such that~\eqref{eq:plus1} holds for all $(\tau_{d-3},\tau_{d-2})\in\FF_2^2$, as desired.
      
      Take an arbitrary $(\tau_{d-3},\tau_{d-2})\in \FF_2^2$. For the chosen $(\tau_2,\ldots,\tau_{d-4},\tau_{d-1},\tau_d,\tau_{d+1})\in\FF_2^{d-2}$ and $\bm{\gamma}\in\FF_2^m$, we have  
      \[
            f_0(\bm{\gamma}^{z_\tau})=f_0(\bm{\gamma})+1 \ \text{ and }\  
            f_{-1}(\bm{\gamma}^{z_\tau})=f_{-1}(\bm{\gamma})+1.
      \]
      Moreover, for this $\bm{\tau}$, we let $j_t$ be as in~\eqref{eq:j_t}, and let $(\beta_1,\ldots,\beta_k)=\bm{\gamma}^\psi$. Then since $z_\tau=(1,\ldots,1,g_{j_{d-d'+1}}^{(d-d'+1)},1,\ldots,1)$, where $g_{j_{d-d'+1}}^{(d-d'+1)}$ appears at the $j_{d-d'+1}$-th coordinate, we have
      \begin{align*}
            (\beta_1,\ldots,\beta_k)^{z_\tau}&=(\beta_1,\ldots,\beta_k)^{(1,\ldots,1,g_{j_{d-d'+1}}^{(d-d'+1)},1,\ldots,1)}\\
            &=(\beta_1,\ldots,\beta_{j_{d-d'+1}-1},\beta_{j_{d-d'+1}}^{g_{j_{d-d'+1}}^{(d-d'+1)}},\beta_{j_{d-d'+1}+1},\ldots,\beta_k).
      \end{align*}
      This implies that for each $j\in\{1,\ldots,k\}\setminus\{j_{d-d'+1}\}$, the set $\Delta^{j-1}\times\{\beta_{j}\}\times\Delta^{k-j}$ contains both $(\beta_1,\ldots,\beta_k)$ and $(\beta_1,\ldots,\beta_k)^{z_\tau}$, and so the set $(\Delta^{j-1}\times\{\beta_{j}\}\times\Delta^{k-j})^{\psi^{-1}}$ contains both $\bm{\gamma}$ and $\bm{\gamma}^{z_\tau}$. Therefore for each $j\in\{1,\ldots,k\}\setminus\{j_{d-d'+1}\}$, 
      \begin{align*}
            &|\{f_0(\bm{\alpha})\mid \bm{\alpha}^\psi\in \Delta^{j-1}\times\{\beta_{j}\}\times\Delta^{k-j}\}|\geq|\{f_0(\bm{\gamma}),f_0(\bm{\gamma}^{z_\tau})\}|=|\{f_0(\bm{\gamma}),f_0(\bm{\gamma})+1\}|=2,\\
            &|\{f_{-1}(\bm{\alpha})\mid \bm{\alpha}^\psi\in \Delta^{j-1}\times\{\beta_{j}\}\times\Delta^{k-j}\}|\geq|\{f_{-1}(\bm{\gamma}),f_{-1}(\bm{\gamma}^{z_\tau})\}|=|\{f_{-1}(\bm{\gamma}),f_{-1}(\bm{\gamma})+1\}|=2.
      \end{align*}
      Then we conclude from the Remark of Definition~\ref{def:determine} that  $j\notin E(f_0)\cup E(f_{-1})$ for each $j\in \{1,\ldots,k\}\setminus \{j_{d-d'+1}\}$. Noticing by Lemma~\ref{lm:determine} that $j_0\in E(f_0)$ and $j_{-1}\in E(f_{-1})$, we obtain $j_0=j_{d-d'+1}=j_{-1}$. This combined with Lemma~\ref{lm:module} implies that $j_t=j_0$ for all $t\in\Z$, which completes the proof.
\end{proof}

\begin{proof}[Proof of Proposition~\ref{prop:exwr}]
      Suppose that $\langle R(H),x_\tau\rangle$ is a primitive subgroup of $\Sym(\Delta)\wr S_k$ acting on $\Omega=\Delta^k$ with product action, where $|\Delta|=2^\ell\geq 5$ and $k\ell=m$.
      Take an arbitrary $(\tau_2,\ldots,\tau_{d-4},\tau_{d+1})\in\FF_2^{d-4}$. We see from Lemma~\ref{lm:not determine} that there exists $(\tau_{d-1},\tau_d)\in\FF_2^2$ such that for each $(\tau_{d-3},\tau_{d-2})\in \FF_2^2$, the index $j_t$ defined in~\eqref{eq:j_t} satisfies $j_{-d'+1}=\dots=j_{d-d'+1}$. Then we derive from Lemma~\ref{lm:determine} that $j_0\in E(f_t)$ for each $t\in\{-d'+1,\ldots,d-d'\}$. Hence, given any $\beta_{j_0}\in\Delta$, it follows from Definition~\ref{def:determine} that there exists a $\{0,1\}$-sequence $(\gamma_t)_{t=-d'+1}^{d-d'}$ such that for each $t\in\{-d'+1,\ldots, d-d'\}$,
      \[
            \{f_t(\bm{\alpha})\mid \bm{\alpha}^\psi\in \Delta^{j_0-1}\times\{\beta_{j_0}\}\times\Delta^{k-j_0}\}=\{\gamma_t\}.
      \]
      Hence, each element of $(\Delta^{j_0-1}\times\{\beta_{j_0}\}\times\Delta^{k-j_0})^{\psi^{-1}}$ is a common solution to the equations 
      \[
            f_t(\bm{\alpha})=\gamma_t\ \text{ for } \  t\in \{-d'+1,\ldots,d-d'\}.
      \]
      Then it follows from Lemma~\ref{lm:number of solutoins} that
      \[
            |(\Delta^{j_0-1}\times\{\beta_{j_0}\}\times\Delta^{k-j_0})^{\psi^{-1}}|\leq 2^{d'-2}.
      \]
      However, since $m=k\ell$ with $k\geq 2$ and $m=d+d'\geq 2d'+d'=3d'$, we obtain
      \[
      |(\Delta^{j_0-1}\times\{\beta_{j_0}\}\times\Delta^{k-j_0})^{\psi^{-1}}|=|\Delta|^{k-1}=2^{\ell(k-1)}=2^{m-\ell}\geq 2^{m/2}>2^{d'-2},
      \]
      a contradiction.
\end{proof}

\section{Proof of the main theorem}\label{sec:proof}

We apply the results in Sections~\ref{sec:primitive}--\ref{sec:PA} to complete the proof of Theorem~\ref{thm:regular}.

\begin{proof}[Proof of Theorem~\ref{thm:regular}]
      Let $H=\langle a_1,\ldots,a_m\rangle$ be a tightly concentric group. Then both~\eqref{enu:C1} and~\eqref{enu:C2} hold, and $3d>2m$. As a consequence, $d-d'=2d-m>m/3$. Adopt the notation in Section~\ref{sec:preliminary}. In particular, $\tau_1=1$, and $\tau_{d+2}=\dots=\tau_m=0$.
      
      By Proposition~\ref{prop:primitive2}, we may choose $\tau_{d+1}\in\FF_2$ such that the group $\langle R(H),x_\tau\rangle$ is primitive on $H$ for each $(\tau_2,\ldots,\tau_d)\in\FF_2^{d-1}$. Note that the right hand side of~\eqref{eq:sum} lies in the algebra $\FF_2[\tau_2,\ldots,\tau_{d'-1}]$ and the inequality $d-d'>m/3\geq 3$ implies $d-4\geq d'$. Thus, we may choose $(\tau_2,\ldots,\tau_{d-4})\in\FF_2^{d-5}$ such that~\eqref{eq:sum} holds. Then, we deduce from Proposition~\ref{prop:exwr} that for the chosen $(\tau_2,\ldots,\tau_{d-4},\tau_{d+1})\in\FF_2^{d-4}$, we can further choose $(\tau_{d-1},\tau_d)\in\FF_2^2$ such that for any $(\tau_{d-3},\tau_{d-2})\in\FF_2^2$, the group $\langle R(H),x_\tau\rangle$ is not a subgroup of $\Sym(\Delta)\wr S_k$ for any set $\Delta$ and positive integer $k$ with $|\Delta|=2^{m/k}\geq 5$. Moreover, we obtain from Proposition~\ref{prop:affine} that for the chosen $(\tau_2,\ldots,\tau_{d-4},\tau_{d-1},\tau_d,\tau_{d+1})\in\FF_2^{d-2}$, there exists $(\tau_{d-3},\tau_{d-2})\in\FF_2^2$ such that $\langle R(H),x_\tau\rangle$ is not contained in any affine group on $H$. To sum up, there exists 
      $(\tau_2,\ldots,\tau_{d+1})\in\FF_2^d$ such that the following conditions hold:
      \begin{itemize}
            \item the group $\langle R(H),x_\tau\rangle$ is primitive on $H$.
            \item the group $\langle R(H),x_\tau\rangle$ is not a subgroup of $\Sym(\Delta)\wr S_k$ for any set $\Delta$ and positive integer $k$ with $|\Delta|=2^{m/k}\geq 5$.
            \item the group $\langle R(H),x_\tau\rangle$ is not contained in any affine group on $H$.
      \end{itemize}
      Combining these conclusions with Lemmas~\ref{lm:SubAlt} and~\ref{lm:AS and Dia}, we obtain from the O'Nan-Scott Theorem~\cite[Page~66]{P1990} that $\langle R(H),x_\tau\rangle=\Alt(H)$ for such $\tau$. Then, by the general framework introduced after Conjecture~\ref{conj:1.8}, we conclude that every tightly concentric group is a $4$-HAT-stabilizer, as desired.
\end{proof}

\noindent\textbf{Acknowledgements.}
The first and third authors acknowledge the support of ARC Discovery Project DP250104965. The second author was supported by the Melbourne Research Scholarship provided by The University of Melbourne.

\end{document}